\def\marker{\>\hbox{${\vcenter{\vbox{
    \hrule height 0.4pt\hbox{\vrule width 0.4pt height 6pt
    \kern6pt\vrule width 0.4pt}\hrule height 0.4pt}}}$}\>}

\documentclass[12pt]{article}

\usepackage[left=2cm,right=2cm,top=3cm,bottom=3cm]{geometry}

\usepackage{amsmath, amssymb, latexsym, amsthm}
\usepackage{fullpage}
\usepackage[colorinlistoftodos]{todonotes}
\usepackage{algorithmic}
\usepackage{algorithm}
\usepackage{slashbox}

\usepackage{tikz}
\usetikzlibrary{shapes,decorations.pathreplacing}

\newtheorem{theorem}{Theorem} 
\newtheorem{theorem*}{Theorem} 
\newtheorem{proposition}[theorem]{Proposition} 

\newtheorem{corollary}[theorem]{Corollary}
\newtheorem{lemma}[theorem]{Lemma}
\newtheorem{conjecture}[theorem]{Conjecture}

\newtheorem{alg}[theorem]{Algorithm}

\theoremstyle{definition}

\newtheorem{question}{Question}

\newtheorem{construction}{Construction}

\theoremstyle{remark}

{\end{enumerate}}

\newcommand{\CL}[1]{\left\lceil #1 \right\rceil}

\newcommand{\FL}[1]{\left\lfloor #1 \right\rfloor}

\newcommand{\proj}{\operatorname{proj}}

\title{The Unit Bar Visibility Number of a Graph}

\author{Emily Gaub$^{1,4}$, Michelle Rose$^{2,4}$, and Paul S.\ Wenger$^{3,4}$}

\date{\today}

\begin{document}
\maketitle

\begin{abstract}
A {\it $t$-unit-bar representation} of a graph $G$ is an assignment of sets of at most $t$ horizontal unit-length segments in the plane to the vertices of $G$ so that (1) all of the segments are pairwise nonintersecting, and (2) two vertices $x$ and $y$ are adjacent if and only if there is a vertical channel of positive width connecting a segment assigned to $x$ and a segment assigned to $y$ that intersects no other segment.
The {\it unit bar visibility number} of a graph $G$, denoted $ub(G)$, is the minimum $t$ such that $G$ has a $t$-unit-bar visibility representation.
Our results include a linear time algorithm that determines $ub(T)$ when $T$ is a tree, bounds on $ub(K_{m,n})$ that determine $ub(K_{m,n})$ asymptotically when $n$ and $m$ are asymptotically equal, and bounds on $ub(K_n)$ that determine $ub(K_n)$ exactly when  $n\equiv 1,2\pmod 6$.

\end{abstract}

\footnotetext[1]{Pacific Univ., Forest Grove, OR 97116; {\tt gaub2680@pacificu.edu}.}
\footnotetext[2]{
Mount St. Mary's Univ., Emmitsburg, MD 21727; {\tt MMRose1@email.msmary.edu}.}
\footnotetext[3]{
School of Mathematical Sciences, Rochester Institute of Technology, Rochester, NY 14623; {\tt pswsma@rit.edu}.}
\footnotetext[4]{
Research supported by REU in Extremal Graph Theory and Dynamical Systems -- NSF award \#1358583.}

\baselineskip18pt 
\section{Introduction}

Motivated by the challenges of very-large-scale integration (VLSI), significant study has focused on graphs with representations having nice geometric descriptions.
A {\it visibility representation} of a graph $G$ consists of an assignment of pairwise disjoint geometric objects (typical examples include segments, Cartesian products of segments, or spheres) to the vertices of $G$ so that two vertices are adjacent if and only if there is an uninterrupted ``line of sight" (generally a line segment or channel of positive width) joining the objects assigned to those vertices.

A {\it bar visibility representation} of a graph $G$ is an assignment of distinct horizontal line segments (henceforth called {\it bars}) in the plane to the vertices of $G$ so that two vertices are adjacent if and only if there is an uninterrupted channel of positive width that joins the bars corresponding to those vertices.
A graph with a bar visibility representation is a {\it bar visibility graph}.
Tamassia and Tollis~\cite{TT} and Wismath~\cite{Wismath} characterized bar visibility graphs: a graph $G$ is a bar visibility graph if and only if it is planar and there is a planar embedding of $G$ such that all cut-vertices appear on the same face.

In a bar visibility representation of a graph, vertices are allowed to have horizontal bars of arbitrary length.
For the viewpoint of VLSI design, it is reasonable to assume that the size of components in a circuit design have roughly equal size.
A graph is a {\it unit bar visibility graph} if it has a bar visibility representation in which every bar has the same length.
Such a representation is called a {\it unit bar visibility representation}.
For simplicity, throughout the paper we assume that all unit bars have length $1$.

In~\cite{DV}, Dean and Veytsel introduced unit bar visibility graphs and characterized the trees, complete bipartite graphs, and complete graphs that are unit bar visibility graphs.
Wigglesworth~\cite{wigglesworth} continued the study of unit bar visibility graphs, characterizing unit-bar visibility graphs that have representations with width at most $2$ (that is, the projection of all bars onto the $x$-axis has length at most $2$). 
Wigglesworth also proved other structural relations between a unit bar visibility graph and its unit bar visibility representations.
To date, there is no characterization of unit bar visibility graphs.

The family of bar visibility graphs is quite restrictive.
In~\cite{CHJLW}, Chang et al.\ introduced a generalization of bar visibility representations that captures all graphs.
A {\it $t$-bar} is the union of (at most) $t$ horizontal bars in the plane.
A graph $G$ has a {\it $t$-bar visibility representation} if there is an assignment of $t$-bars to the vertices of $G$ such that two vertices $u$ and $v$ are adjacent in $G$ if and only if there is an uninterrupted vertical channel of positive width joining a bar assigned to $u$ to a bar assigned to $v$.
The {\it bar visibility number} of $G$, denoted $b(G)$, is the minimum $t$ such that $G$ has a $t$-bar visibility representation.
In~\cite{CHJLW}, Chang et al.\ proved that $b(G)\le 2$ if $G$ is planar, determined the bar visibility number of complete bipartite graphs within $1$, proved that $b(K_n)=\CL{n/6}$, and proved that $b(G)\le \CL{|V(G)|/6}+2$ for all graphs.
The notion of bar visibility numbers has since been extended to directed graphs~\cite{ABHW}.

In this paper we study $t$-bar visibility representations of graphs in which every bar (that is, every individual bar in every $t$-bar) has the same length, which we can assume to be $1$.
We call such a representation a {\it $t$-unit-bar visibility representation}.
The {\it unit bar visibility number} of $G$, denoted $ub(G)$, is the minimum $t$ such that $G$ has a $t$-unit-bar visibility representation.
A $t$-unit-bar visibility representation of $G$ with $t=ub(G)$ is called {\it optimal}.

By assigning the edges of a graph disjoint intervals of length $1$ on the $x$-axis and giving vertices $u$ and $v$ bars that project onto the interval for the edge $uv$, it is clear that $ub(G)\le \Delta(G)$ for all $G$, where $\Delta(G)$ denotes the maximum degree of $G$.
Thus $ub(G)$ is well-defined.
Since a $t$-unit-bar visibility representation of $G$ is also a $t$-bar visibility representation of $G$, it follows that $b(T)\le ub(T)$.

We study the unit bar visibility number of graphs in various families.
In Section~\ref{trees}, we present a linear-time algorithm that determines the unit bar visibility number of trees and generates an optimal representation.
In Section~\ref{bipartite} we give bounds for the unit bar visibility number of complete bipartite graphs that are asymptotically tight when the partite sets are of asymptotically equal sizes.
In Section~\ref{complete} we study the unit bar visibility number of complete graphs, proving that $\lceil \frac{n}{6} \rceil \leq ub(K_n) \leq \lceil \frac{n+4}{6} \rceil$.
Section~\ref{conclusion} contains open questions and conjectures.

Throughout the paper, all graphs are finite.
We let $d_G(v)$ denote the degree of a vertex $v$ in a graph $G$; when the graph is clear, we simply write $d(v)$. 
Given a positive integer $n$, we let $[n]$ denote the set $\{1,\ldots,n\}$.
We let $K_{n}$ denote the complete graph on $n$ vertices and let $K_{m,n}$ denote the complete bipartite graph with partite sets of order $m$ and $m$.

\section{Preliminaries}

When an arrangement of unit bars in the plane is given without it being identified as a representation of a specific graph, we refer to it as a unit bar visibility {\it layout}.
When we refer to a bar in a $t$-unit-bar visibility representation or layout we are referring to one of the unit bars in one of the $t$-unit-bars.
Since all bars are assumed to have length $1$ we may describe each bar by the coordinates of its left endpoint; we denote the left endpoint of a bar $b$ by $(x_b,y_b)$.
If $y_b<y_{b'}$ then we say that $b$ is {\it below} $b'$ and $b'$ is {\it above} $b$.
If $x_b<x_{b'}$ then we say that $b$ is {\it to the left of} $b'$ and $b'$ is {\it to the right of} $b$.

We say that two bars {\it see} each other if there is an uninterrupted vertical channel of positive width between the bars.
When such a channel exists, we also say that the corresponding $t$-unit-bars and the corresponding vertices see each other.
If two bars $b$ and $b'$ see each other and $b$ is below $b'$, then we say that $b$ {\it sees $b'$ above} and that {\it $b'$ sees $b$ below}.

%
%

Let $R$ be a $t$-unit-bar visibility layout.
We define $G(R)$ to be the unit bar visibility graph that is represented by $R$ when all bars represent distinct vertices.
Given a bar $b$ in $R$, define $\proj(b)$ to be the projection of $b$ onto the $x$-axis.
Similarly, for a set of bars $S$, define $\proj(S)$ to be $\bigcup_{b\in S}\proj(b)$.
Two bars $b$ and $b'$ in $R$ are said to {\it overlap} if $\proj(b)$ and $\proj(b')$ have an intersection of positive measure.
The {\it components} of $R$ are the sets of bars in $R$ that correspond to the components of $G(R)$.
Thus distinct components of $R$ have disjoint projections onto the $x$-axis.

Let $R$ and $R'$ be two unit bar visibility layouts.
The {\it disjoint union} of $R$ and $R'$ is the unit bar visibility layout obtained by arranging $R$ and $R'$ in the plane so that $\proj(R)\cap \proj(R')=\emptyset$ and (without loss of generality) $x_b<x_{b'}$ for all $b\in R$ and $b'\in R'$.
With this terminology, every unit bar visibility layout is the disjoint union of its components.
For convenience we label the components of $R$ as $R_1,\ldots,R_{\ell}$ so that for $i<j$ we have that $x_b<x_{b'}$ whenever $b\in R_i$ and $b'\in R_j$.

\begin{lemma}\label{lem:distincty}
Every graph $G$ has an optimal $t$-unit-bar visibility representation in which all bars have distinct $y$-coordinates.
\end{lemma}

\begin{proof}
Let $R$ be an optimal representation of $G$ in which there are the fewest pairs of bars that share their $y$-coordinate.
Suppose $b$ and $b'$ are two bars that share their $y$-coordinate.
Since $G$ is finite, there exists $\epsilon>0$ such that no bar has its $y$-coordinate in $(y_b-\epsilon,y_b)\cup(y_b,y_b+\epsilon)$.
Increase the $y$-coordinate of $b$ by $\epsilon/2$ to obtain the $t$-unit-bar visibility layout $R'$.
It is clear that $R'$ is also a $t$-unit bar visibility representation of $G$, contradicting the minimality of $R$.
\end{proof}

\section{Trees}\label{trees}

A tree that is a unit bar visibility graph is a {\it unit bar visibility tree}; we abbreviate unit bar visibility tree as UBVT.
In~\cite{DV}, Dean and Veytsel characterized unit bar visibility trees.
A {\it caterpillar} is a tree in which all vertices with degree at least $2$ lie on a single path.

\begin{theorem}\label{UBVG trees}(Dean-Veytsel [2003])
A tree $T$ is a UBVT if and only if $\Delta(T)\le 3$ and $T$ is a subdivision of a caterpillar.
\end{theorem}

We present a linear time algorithm that determines the unit bar visibility number of trees.
With a slight modification, this algorithm will generate an optimal $t$-unit-bar visibility representation of a tree.

A {\it unit bar visibility forest} is a graph in which every component is a UBVT.
Given a graph $G$, let the {\it unit bar visibility arboricity} of $G$, denoted $\Upsilon_{ub}(G)$, be the minimum number of unit bar visibility forests needed to decompose $G$.
Note that a decomposition of a graph into $k$ forests gives a decomposition of $G$ into trees such that no vertex is in more than $k$ of those trees.
If $T$ is a tree, then a decomposition of $T$ into trees such that no vertex is in more than $k$ of those trees also yields a decomposition of a $T$ into at most $k$ forests.
A decomposition of a tree $T$ into UBVTs such that no vertex is in more than $\Upsilon_{ub}(T)$ of the UBVTs will be called an {\it optimal} decomposition.
As a first step towards giving an algorithm that determines $ub(T)$ for all trees $T$, we prove that for all trees the unit bar visibility number and unit bar visibility arboricity are equal.

\begin{lemma}\label{lem:walktopath}
Let $T$ be a tree, and let $R$ be a unit bar visibility layout.
Let $f:R\to V(T)$ be a labeling of the bars in $R$ that induces a homomorphism from $G(R)$ to $T$.
If $u,v\in V(T)$ and bars assigned to $u$ and $v$ lie in some component $R'$ of $R$, then for each edge $xy$ on the unique $u,v$-path in $T$ there are bars assigned to $x$ and $y$ in $R'$ that see each other.
\end{lemma}

\begin{proof}
Let $b(u)$ and $b(v)$ be bars in $R'$ that are assigned to $u$ and $v$, respectively.
Since $G(R')$ is connected, there is a sequence of bars $b(u)=b_0,b_1,\ldots,b_\ell=b(v)$ in $R$ such that $b_i$ and $b_{i-1}$ see each other for all $i\in [\ell]$.
Every edge on the unique $u,v$-path in $T$ must lie in the walk $f(b_0),f(b_1),\ldots,f(b_\ell)$.
\end{proof}

\begin{theorem}\label{thm:treesarboricity}
If $T$ is a tree, then $ub(T)=\Upsilon_{ub}(T)$.
\end{theorem}

\begin{proof}
First suppose that $\Upsilon_{ub}(T)=k$, and let $\{T_1,\ldots,T_{\ell}\}$ be a decomposition of $T$ into UBVTs such that each vertex in $T$ is in at most $k$ trees in the decomposition.
The disjoint union of unit bar visibility representations of $T_1,\ldots,T_{\ell}$ is a $k$-unit-bar visibility representation of $T$, and therefore $ub(T)\le \Upsilon_{ub}(T)$.

We now show that $\Upsilon_{ub}(T)\le ub(T)$.
By Lemma~\ref{lem:distincty}, we know that there is an optimal $t$-unit-bar visibility representation $R$ of $T$ such that every bar in $R$ has a distinct $y$-coordinate.
Among all such representations of $T$, let $R$ be chosen so that it has the minimum number of pairs of bars $\{b,b'\}$ such that $b$ and $b'$ correspond to the same vertex and lie in the same component of $R$.

For each $v\in V(T)$, let $b_1(v),\ldots,b_t(v)$ be the $t$ bars assigned to $v$.
First assume that $b_i(v)$ and $b_j(v)$ are in distinct components of $R$ for all $v\in V(T)$ and all $i,j\in [t]$ such that $i\neq j$.
Let $R_1$ be the first component of $R$.
Since the labels on the bars in $R_1$ are distinct, the labels induce an isomorphism from $G(R_1)$ to the subgraph of $T$ induced by the vertices assigned to the bars in $R_1$.
Therefore $G(R_1)$ is a tree.
Removing $R_1$ from $R$ and applying induction then shows that $R$ is the disjoint union of unit bar visibility representations of trees.
Therefore $\Upsilon_{ub}(T)\le t$.

Now let $\{b(v),b'(v)\}$ be a pair of bars that are assigned to the same vertex and lie in the same component of $R$.
Without loss of generality, assume that $y_{b(v)}<y_{b'(v)}$.
Fix a value $z$ such that $y_{b(v)}<z<y_{b'(v)}$ and $z$ is not the $y$-coordinate of any bar in $R$.
Place the horizontal line $y=z$ through $R$.
This immediately partitions $R$ into two sets: $A$, the set of bars above the line $y=z$; and $B$, the set of bars below the line $y=z$.
Further partition $A$ into $A_1,\ldots,A_\ell$, the components of the unit bar visibility layout consisting of just the bars in $A$.
Similarly, partition $B$ into $B_1,\ldots,B_{\ell'}$, the components of the layout consisting of just the bars in $B$.
Observe that $\mathcal P=\{A_1,\ldots,A_\ell,B_1,\ldots,B_{\ell'}\}$ is partition of $R$ (see the top left picture in Figure~\ref{fig1}).

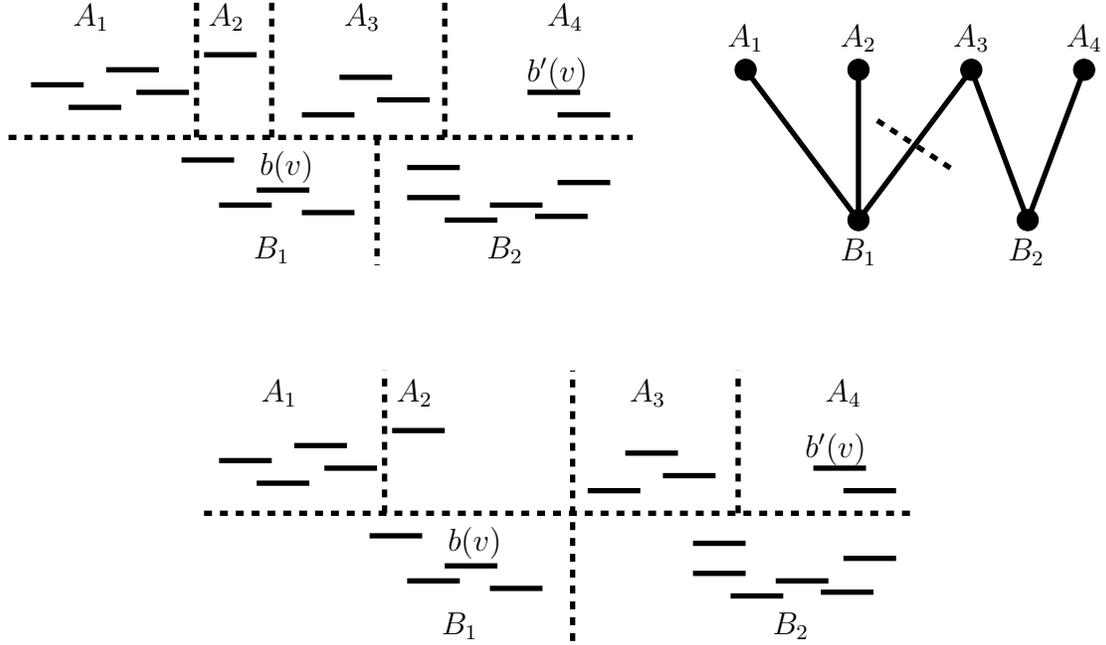
\begin{figure}
\centering
\begin{tikzpicture}
\draw[dashed, line width=2pt] (-.3,.1)--(8,.1);
\draw[dashed, line width=2pt] (2.2,.1)--(2.2,2);
\draw[dashed, line width=2pt] (3.2,.1)--(3.2,2);
\draw[dashed, line width=2pt] (4.6,.1)--(4.6,-1.6);
\draw[dashed, line width=2pt] (5.5,.1)--(5.5,2);
\draw[line width=2pt] (.5,.5)--(1.2,.5);
\draw[line width=2pt] (1.4,.7)--(2.1,.7);
\draw[line width=2pt] (0,.8)--(.7,.8);
\draw[line width=2pt] (1,1)--(1.7,1);
\node at (.8,1.7) {$A_1$};
\draw[line width=2pt] (2.3,1.2)--(3,1.2);
\node at (2.6,1.7) {$A_2$};
\draw[line width=2pt] (2,-.2)--(2.7,-.2);
\draw[line width=2pt] (2.5,-.8)--(3.2,-.8);
\draw[line width=2pt] (3,-.6)--(3.7,-.6);
\node [above] at (3.4,-.65) {$b(v)$};
\draw[line width=2pt] (3.6,-.9)--(4.3,-.9);
\node at (3.2,-1.4) {$B_1$};
\draw[line width=2pt] (3.6,.4)--(4.3,.4);
\draw[line width=2pt] (4.1,.9)--(4.8,.9);
\draw[line width=2pt] (4.6,.6)--(5.3,.6);
\node at (4.4,1.7) {$A_3$};
\draw[line width=2pt] (5,-.3)--(5.7,-.3);
\draw[line width=2pt] (5,-.7)--(5.7,-.7);
\draw[line width=2pt] (5.5,-1)--(6.2,-1);
\draw[line width=2pt] (6.1,-.8)--(6.8,-.8);
\draw[line width=2pt] (6.7,-.95)--(7.4,-.95);
\draw[line width=2pt] (7,-.5)--(7.7,-.5);
\node at (6.3,-1.4) {$B_2$};
\draw[line width=2pt] (6.6,.7)--(7.3,.7);
\draw[line width=2pt] (7,.4)--(7.7,.4);
\node [above] at (7,.6) {$b'(v)$};
\node at (7.1,1.7) {$A_4$};
\draw[fill] (9.5,1) circle (4pt);
\node [above] at (9.5,1.1) {$A_1$};
\draw[fill] (11,1) circle (4pt);
\node [above] at (11,1.1) {$A_2$};
\draw[fill] (12.5,1) circle (4pt);
\node [above] at (12.5,1.1) {$A_3$};
\draw[fill] (14,1) circle (4pt);
\node [above] at (14,1.1) {$A_4$};
\draw[fill] (11,-1) circle (4pt);
\node [below] at (11,-1.1) {$B_1$};
\draw[fill] (13.25,-1) circle (4pt);
\node [below] at (13.23,-1.1) {$B_2$};
\draw[line width=2pt] (9.5,1)--(11,-1)--(11,1);
\draw[line width=2pt] (11,-1)--(12.5,1)--(13.25,-1)--(14,1);
\draw[dashed, line width=2pt] (11.25,.325)--(12.25,-.325);
\draw[line width=2pt] (3.,-4.5)--(3.7,-4.5);
\draw[line width=2pt] (3.9,-4.3)--(4.6,-4.3);
\draw[line width=2pt] (2.5,-4.2)--(3.2,-4.2);
\draw[line width=2pt] (3.5,-4)--(4.2,-4);
\node at (3.3,-3.3) {$A_1$};
\draw[line width=2pt] (4.8,-3.8)--(5.5,-3.8);
\node at (5.1,-3.3) {$A_2$};
\draw[line width=2pt] (4.5,-5.2)--(5.2,-5.2);
\draw[line width=2pt] (5.0,-5.8)--(5.7,-5.8);
\draw[line width=2pt] (5.5,-5.6)--(6.2,-5.6);
\node [above] at (5.9,-5.65) {$b(v)$};
\draw[line width=2pt] (6.1,-5.9)--(6.8,-5.9);
\node at (5.7,-6.4) {$B_1$};
\draw[line width=2pt] (7.4,-4.6)--(8.1,-4.6);
\draw[line width=2pt] (7.9,-4.1)--(8.6,-4.1);
\draw[line width=2pt] (8.4,-4.4)--(9.1,-4.4);
\node at (8.2,-3.3) {$A_3$};
\draw[line width=2pt] (8.8,-5.3)--(9.5,-5.3);
\draw[line width=2pt] (8.8,-5.7)--(9.5,-5.7);
\draw[line width=2pt] (9.3,-6)--(10,-6);
\draw[line width=2pt] (9.9,-5.8)--(10.6,-5.8);
\draw[line width=2pt] (10.5,-5.95)--(11.2,-5.95);
\draw[line width=2pt] (10.8,-5.5)--(11.5,-5.5);
\node at (10.1,-6.4) {$B_2$};
\draw[line width=2pt] (10.4,-4.3)--(11.1,-4.3);
\node [above] at (10.7,.-4.4) {$b'(v)$};
\draw[line width=2pt] (10.8,-4.6)--(11.5,-4.6);
\node at (10.8,-3.3) {$A_4$};
\draw[dashed, line width=2pt] (2.3,-4.9)--(11.7,-4.9);
\draw[dashed, line width=2pt] (4.7,-4.9)--(4.7,-3);
\draw[dashed, line width=2pt] (7.2,-6.6)--(7.2,-3);
\draw[dashed, line width=2pt] (9.4,-4.9)--(9.4,-3);

\end{tikzpicture}
\caption{Top left: A representation of a tree $T$ with two bars assigned to $v$ in the same component, along with a partition of the representation.
Top right: The auxiliary bipartite graph.
Bottom: A new representation of $T$ with $b(v)$ and $b'(v)$ in different components.}\label{fig1}
\end{figure}

Form an auxiliary graph $H$ with $V(H)=\mathcal P$ where two elements of $\mathcal P$ are adjacent if bars from those elements can see each other in $R$ (see the top right picture in Figure~\ref{fig1}).
It follows that $H$ is a bipartite interval graph.
Hence $H$ is a forest with a bipartition corresponding to the partition of $R$ into $A$ and $B$.
Let $B_i$ be the set that contains $b(v)$ and let $A_{i'}$ be the set that contains $b'(v)$.
Let $e=B_iA_{j}$ be the first edge on the path in $H$ from $B_i$ to $A_{i'}$.
Let $R'$ be the union of the elements of $\mathcal P$ in the component of $H-e$ that contains $A_{i'}$.
Finally, let $S$ be the $t$-unit-bar visibility layout obtained from the disjoint union of $R\setminus R'$ and $R'$.
You can picture $S$ as being obtained by ``sliding" the bars in $R'$ to the right of all other bars in $R$ (see the bottom picture in Figure~\ref{fig1}).

We claim that $S$ is a $t$-unit-bar visibility representation of $T$, contradicting the minimality of $R$ since $b(v)$ and $b'(v)$ are no longer in the same component.
It is clear that there are no bars that see each other in $S$ that did not see each other in $R$.
Similarly, there are no pairs of bars assigned to the same vertex lying in the same component of $S$ that do not also lie in the same component of $R$.
Therefore it remains to show that every edge in $T$ is a visibility in $S$.
If two bars $b(w)$ and $b(w')$ can see each other in $R$ but cannot see each other in $S$, then without loss of generality we can assume that $b(w)\in B_i$ and $b(w')\in A_{j}$.
Both $b(w)$ and $b(w')$ lie in components of $S$ that contain bars corresponding to $v$.
Therefore there is a component of $S$ that contains a $w,v$-walk in $T$ and a component that contains a $w',v$-walk in $T$.
One of these walks must contain the edge $ww'$, and by Lemma~\ref{lem:walktopath}, the corresponding component of $S$ contains the edge $w'w$.
Therefore $S$ contains every edge of $T$, contradicting the minimality of $R$.
\end{proof}

An immediate corollary of Theorems~\ref{UBVG trees} and~\ref{thm:treesarboricity} is the following lower bound on the unit bar visibility number of a tree.

\begin{corollary}\label{cor:treelowerbound}
If $T$ is a tree, then $ub(T)\ge \CL{\Delta(T)/3}$.
\end{corollary}

\begin{proof}
Each UBVT has maximum degree at most $3$, and therefore a vertex in $T$ of degree $d$ must be in at least $\CL{d/3}$ elements of a decomposition of $T$ into UBVTs.
\end{proof}

We now prove our first upper bound on the unit bar visibility number of trees.

\begin{theorem}\label{thm:treebound}
If $T$ is a tree, then there is a decomposition of $T$ into UBVTs such that each  vertex $v\in V(G)$ is in at most $\CL{\frac{d(v)+1}{3}}$ elements of the decomposition.
Therefore $ub(T)\le \CL{\frac{\Delta(T)+1}{3}}$.
\end{theorem}

\begin{proof}
We proceed by induction on the number of vertices in $T$.
For the base case, we decompose $K_{1,n}$ into $\FL{n/3}$ copies of $K_{1,3}$ and, if $n$ is not divisible by $3$, one copy of $K_{1,r}$ where $1\le r\le 2$ and $r\equiv n\pmod 3$.

Now assume that $T$ is not a star and let $v$ be a vertex in $T$ with exactly one neighbor that is not a leaf.
If $d(v)\ge 4$, then let $T'$ be obtained by deleting three neighbors of $v$ that are leaves.
The addition of a copy of $K_{1,3}$ to the decomposition of $T'$ from the inductive hypothesis yields the desired decomposition of $T$.
If $d(v)=3$, then let $T'$ be obtained by deleting the two neighbors of $v$ that are leaves.
The addition of a copy of $K_{1,2}$ to the decomposition of $T'$ from the inductive hypothesis yields the desired decomposition of $T$.
If $d(v)=2$, then let $v'$ be the leaf neighbor of $v$ and let $T'=T-v'$.
By induction, $T'$ has a decomposition into UBVTs where $v$ lies in exactly one UBVT.
By adding $v'$ and the edge $vv'$ to the element of the decomposition that contains $v$, we obtain the desired decomposition of $T$.
\end{proof}

Corollary~\ref{cor:treelowerbound} and Theorem~\ref{thm:treebound} actually determine the unit bar visibility number of any tree whose maximum degree is not a multiple of $3$.
However, both Corollary~\ref{cor:treelowerbound} and Theorem~\ref{thm:treebound} can be sharp when $\Delta(T)$ is a multiple of $3$.
As a simple example, given a tree $T$ with maximum degree $3$, $ub(T)=1$ if $T$ is a subdivision of a caterpillar and $ub(T)=2$ otherwise.
While it is not difficult to show that both bounds are also sharp for larger maximum degrees that are multiples of $3$, we choose to omit such constructions for brevity.

We now present a linear-time algorithm called UNIT\textunderscore BAR\textunderscore TREE that determines if $ub(T)=\CL{\Delta(T)/3}$.
The algorithm decomposes trees into UBVTs, and in Theorem~\ref{thm:TreeAlg} we prove that each vertex is in at most $\Upsilon_{ub}(T)$ elements in the decomposition.
Therefore, by Theorem~\ref{thm:treesarboricity}, the algorithm determines $ub(T)$.

Given a tree $T$, UNIT\textunderscore BAR\textunderscore TREE$(T)$ selects a root of $T$ and performs a post-order traversal of $T$, calling two subroutines, called PRUNE and COLOR, at each vertex.
When PRUNE is called at a vertex $v$ that is not the root, a maximum of $\CL{\frac{\Delta(G)}{3}}-1$ UBVTs containing $v$ and its descendants are added to the decomposition of $T$.
When PRUNE is called at the root, up to $\CL{\frac{\Delta(G)}{3}}$ UBVTs containing the root may be added to the decomposition.
After PRUNE is called at $v$, the algorithm calls COLOR at $v$.
If $v$ is not the root, COLOR either assigns a color to the edge joining $v$ and its parent, or halts and declares $ub(T)=\CL{\frac{\Delta(G)+1}{3}}$.
If $v$ is the root, COLOR declares $ub(T)=\CL{\frac{\Delta(G)}{3}}$ or $ub(T)=\CL{\frac{\Delta(G)+1}{3}}$.
The decision to assign a particular color or halt depends complexity of the portion of the tree that remains below $v$.

We now describe the PRUNE routine in detail.
When PRUNE is called at $v$, COLOR has already been called at all children of $v$ without halting.
Consequently each edge joining $v$ to a child has been colored green, yellow, or red (as explained in the COLOR routine below).
Let $T'$ be the tree that remains when PRUNE is called at $v$.
The PRUNE routine selects subtrees of $T'$ rooted at $v$ and removes their edges and vertices (excluding $v$) from $T'$; we say that these subtrees are {\it pruned} from $T'$.
For each child $x$ of $v$ in $T'$, let $T'_{vx}$ denote the subtree of $T'$ consisting of $v$, $x$, and all descendants of $x$ in $T'$.
The choices of the subtrees rooted at $v$ are made greedily according to the following priority ordering of edge colors joining $v$ to its children (this ordering is summarized in Table~\ref{tab:prefCuts}).
In the following list, let $x$, $x'$, and $x''$ be children of $v$, though it is possible that $v$ does not have three children.
\begin{enumerate}
\item If $vx$ is red and $vx'$ is green, then prune $T'_{vx}\cup T'_{vx'}$.
\item If $vx$ is red, then prune $T'_{vx}$.
\item If $vx$ and $vx'$ are yellow and $vx''$ is green, then prune $T'_{vx}\cup T'_{vx'}\cup T'_{vx''}$.
\item If $vx$ and $vx'$ are yellow, then prune $T'_{vx}\cup T'_{vx'}$.
\item If $vx$ is yellow and $vx'$ and $vx''$ are green, then prune $T'_{vx}\cup T'_{vx'}\cup T'_{vx''}$.
\item If $vx$ is yellow and $vx'$ is green, then prune $T'_{vx}\cup T'_{vx'}$.
\item If $vx$ is yellow, then prune $T'_{vx}$.
\item If $vx$, $vx'$ and $vx''$ are green, then prune $T'_{vx}\cup T'_{vx'}\cup T'_{vx''}$.
\item If $vx$ and $vx'$ are green, then prune $T'_{vx}\cup T'_{vx'}$.
\item If $vx$ is green, then prune $T'_{vx}$.
\end{enumerate}
If $v$ is not the root of $T$, then PRUNE stops when $v$ has no remaining children or $\CL{\frac{\Delta(T)}{3}}-1$ trees have been pruned at $v$.
If $v$ is the root, then PRUNE stops when $v$ has no remaining children or $\CL{\frac{\Delta(T)}{3}}$ trees have been pruned at $v$.

\begin{table}[h]
\centering
\begin{tabular}{|r c l|}
	\hline
  Most preferred& &Least preferred \\[2pt] \hline
  RG & $\succ$ R $\succ$ YYG $\succ$ YY $\succ$ YGG $\succ$ YG $\succ$ Y $\succ$ GGG $\succ$ GG $\succ$& G
\\	\hline
\end{tabular}
\caption{\label{tab:prefCuts}Preferred pruning order in PRUNE, and ordering of color types in the proof of Theorem~\ref{thm:TreeAlg}.}
\end{table}

We now describe COLOR in detail.
First, assume that $v$ is not the root of $T$.
Let $T''$ be the subtree of $T$ that remains after PRUNE is called at $v$, and let $u$ be the parent of $v$.
The COLOR routine assigns a color to $uv$ or halts and declares $ub(T)=\CL{\frac{\Delta(G)+1}{3}}$ as follows  (summarized in Table~\ref{tab:remEdges}).

\begin{center}
\begin{tabular}{rll}

1)& Color $uv$ green (G) if&  a) $v$ has no children in $T''$, or \\&& b) $v$ has exactly one child $x$ in $T'$ and $vx$ is green.\\ &&\\
2)& Color $uv$ yellow (Y) if & a) $v$ has exactly two children $x$ and $x'$ in $T''$, \\&&\qquad \qquad and $vx$ and $vx'$ are both green, or\\
&& b) $v$ has exactly one child $x$ in $T''$ and $vx$ is yellow, or\\
&&c) $v$ has exactly two children $x$ and $x'$ in $T''$, \\&&\qquad \qquad and $vx$ is yellow and $vx'$ is green.\\
&&\\
3)& Color $uv$ red (R) if & a) $v$ has exactly two children $x$ and $x'$ in $T''$, \\&&\qquad \qquad and $vx$ and $vx'$ are both yellow, or\\
&& b) $v$ has exactly one child $x$ in $T''$ and $vx$ is red.\\
&&\\
4)& Declare $ub(T)=\CL{\frac{\Delta(G)+1}{3}}$ if & a) $v$ has at least three children in $T''$, or\\
&& b) $v$ has two children $x$ and $x'$ in $T''$ and $vx$ is red.\\
\end{tabular}
\end{center}

\begin{table}
\centering
\begin{tabular}{ |c | c| }
	\hline
  Edges to remaining children at $v$ after pruning & Action \\[2pt] \hline
  None or G & color $uv$ G \\
  GG, Y, or YG & color $uv$ Y \\
  YY or R & color $uv$ R \\
  At least three, RR, RY, or RG & declare $ub(T)=\CL{\frac{\Delta(G)+1}{3}}$\\

  	\hline
\end{tabular}
\caption{\label{tab:remEdges}Action of COLOR$(v)$, where $u$ denotes the parent of $v$.}
\end{table}
If $v$ is the root of $T$, then COLOR declares $ub(T)=\CL{\frac{\Delta(T)}{3}}$ if $v$ has no children after PRUNE runs at $v$, and declares $ub(T)=\CL{\frac{\Delta(T)+1}{3}}$ otherwise.

We now present UNIT\textunderscore BAR\textunderscore TREE(T) in pseudocode.

\begin{alg}[UNIT\textunderscore BAR\textunderscore TREE(T)]\label{alg:tree}
\qquad 

Input: A tree $T$.

Output: $ub(T)$.  If $ub(T)=\CL{\frac{\Delta(T)}{3}}$, then also a decomposition $\mathcal T$ of $T$ into UBVTs so that every vertex is in at most $\CL{\Delta(T)/3}$ trees in $\mathcal T$.

1. Initialize: Choose a vertex $r$ in $T$ and let $T_r$ be $T$ rooted at $r$. Set $\mathcal T=\emptyset$.

2. Do a postorder traversal of $T_r$.
Let $v$ be the current vertex.

\quad a. Run PRUNE$(v)$ and add each UBVT pruned at $v$ to $\mathcal T$.

\quad b. Run COLOR$(v)$.

3. Return $ub(T)$.
If $ub(T)=\CL{\Delta(T)/3}$, also return $\mathcal T$.
\end{alg}

Before proving that Algorithm~\ref{alg:tree} produces an optimal decomposition of $T$, we present a lemma that allows us to modify UBVTs rooted at the same vertex.
Given a tree $T$ with root $v$, we call each maximal subtree of $T$ that contains $v$ as a leaf a {\it branch} of $T$.
We assign a color to each branch $B$ in $T$ depending on its structure as described below.
The color that is assigned to $B$ is the same as the color that the COLOR routine would assign to the edge in $B$ that is incident to $v$.
\begin{enumerate}
\item $B$ is red (R) if $B$ contains vertices of degree $3$, all of which lie in a single path, and no such path also contains $v$ (i.e.\ $B$ is a subdivided caterpillar with maximum degree $3$ and $v$ is not on the spine);
\item $B$ is yellow (Y) if $B$ contains vertices of degree $3$, all of which lie in a single path that also contains $v$ (i.e.\ $B$ is a subdivided caterpillar with maximum degree $3$ and $v$ is on the spine);
\item $B$ is green (G) if $B$ contains no vertices of degree $3$ (i.e.\ $B$ is a path).
\end{enumerate}
The multiset of the colors of the branches of $T$ at $v$ is the {\it color-type} of $T$, denoted $c(T)$.
For convenience we suppress set notation and record each color-type as a string of Rs followed by Ys followed by Gs.
It is clear that a tree is a UBVT if and only if its color-type is RG, R, YYG, YY, YGG, YG, Y, GGG, GG, or G.
We rank these strings according to the lexicographic ordering arising from the ordering $\textrm R\succ \textrm Y\succ \textrm G$.
This ranking is shown in Table~\ref{tab:prefCuts}.

Let $T_1$ and $T_2$ be two trees with a common root $v$, and let $B_1$ and $B_2$ be branches of $T_1$ and $T_2$ respectively.
A {\it branch-swap} of $B_1$ and $B_2$ exchanges $B_1$ and $B_2$ yielding the trees $(T_1-B_1)\cup B_2$ and $(T_2-B_2)\cup B_1$.
We say that $T_1$ {\it absorbs} $B_2$ and $T_2$ {\it gives} $B_2$ if we add $B_2$ to $T_1$ yielding the trees $T_1\cup B_2$ and $T_2-B_2$.

\begin{lemma}\label{lem:treeswap}
Let $T_1$ and $T_2$ be two UBVTs that are rooted at $v$, let $B_1$ be a branch of $T_1$, and let $B_2$ be a branch of $T_2$.
\begin{enumerate}
\item If $B_1$ and $B_2$ have the same color, then the branch-swap of $B_1$ and $B_2$ yields two UBVTs.
\item If $c(T_2)\in \{\textrm{YGG},\textrm{YG},\textrm{GGG},\textrm{GG},\textrm{G}\}$, $B_1$ is yellow, and $B_2$ is green, then the branch-swap of $B_1$ and $B_2$ yields two UBVTs.
\item If $c(T_2)\in \{\textrm{GG},\textrm G\}$, $B_1$ is red, and $B_2$ is green, then the branch-swap of $B_1$ and $B_2$ yields two UBVTs.
\item If $d_{T_1}(v)\le 2$ and $B_2$ is green, then the absorption of $B_2$ by $T_1$ yields two UBVTs.
\item If $d_{T_1}(v)=1$, the only branch in $T_1$ is yellow, and $B_2$ is green or yellow, then the absorption of $B_2$ by $T_1$ yields two UBVTs.
\item Suppose that $B_1$ is red and $c(T_2)\in\{\textrm{GGG},\textrm{GG}\}$.
If $B_2$ and $B_3$ are both branches of $T_2$, then $(T_1-B_1)\cup B_2\cup B_3$ and $(T_2-(B_2\cup B_3))\cup B_1$ are both UBVTs.
\item If $B_1$ is red and $c(T_2)=\textrm{G}$, then $(T_1-B_1)\cup B_2$ and $(T_2-B_2)\cup B_1$ are both UBVTs.
\end{enumerate}
\end{lemma}

\begin{proof}
In all cases, it is trivial to check that the resulting trees have color-types of UBVTs.
\end{proof}

\begin{theorem}~\label{thm:TreeAlg}
Given a tree $T$, UNIT\textunderscore BAR\textunderscore TREE$(T)$ determines $ub(T)$ in time linear in the order of the tree.
Furthermore, if $ub(T)=\CL{\Delta(T)/3}$, then UNIT\textunderscore BAR\textunderscore TREE$(T)$ also generates a decomposition of $T$ into UBVTs so that each vertex is in at most $\CL{\Delta(T)/3}$ of the UBVTs.
\end{theorem}

\begin{proof}
First assume that UNIT\textunderscore BAR\textunderscore TREE$(T)$ returns $ub(T)=\CL{\Delta(T)/3}$.
In this case, UNIT\textunderscore BAR\textunderscore TREE$(T)$ also returns $\mathcal T$, a decomposition of $T$ into UBVTs.
If $v$ is a vertex that is not the root of $T$, then $\mathcal T$ contains at most $\CL{\Delta(T)/3}-1$ UBVTs that are rooted at $v$ and exactly one UBVT containing $v$ that is rooted at an ancestor of $v$.
If $v$ is the root, then $\mathcal T$ contains at most $\CL{\Delta(T)/3}$ UBVTs that contain $v$.
Therefore, if UNIT\textunderscore BAR\textunderscore TREE$(T)$ declares $ub(T)=\CL{\Delta(T)/3}$, then by Theorem~\ref{thm:treesarboricity} $\mathcal T$ serves as a certificate.

For the rest of the proof we assume that there is a tree $T$ such that UNIT\textunderscore BAR\textunderscore TREE$(T)$ returns $ub(T)=\CL{\frac{\Delta(T)+1}{3}}$ while $ub(T)=\CL{\Delta(T)/3}$.
Note that there is only a discrepancy in these two values if $\Delta(T)$ is a multiple of $3$.
Thus we assume that $\Delta(T)$ is a multiple of $3$.
We will write $\frac{\Delta(T)}{3}+1$ for $\CL{\frac{\Delta(T)+1}{3}}$ and we will suppress the ceiling notation in $\CL{\frac{\Delta(T)}{3}}$.
By Theorem~\ref{thm:treesarboricity} we may assume that $T$ has a decomposition into UBVTs so that each vertex is in at most ${\Delta(T)/3}$ of the UBVTs.
Let $\mathcal T=\{T_1,\ldots,T_{\ell}\}$ be the partial decomposition of $T$ into UBVTs when Algorithm~\ref{alg:tree} halts and declares $ub(T)=\frac{\Delta(T)}{3}+1$.
Order $\mathcal T$ so that $T_i$ is pruned before $T_j$ when $i<j$.
Let $\widehat{\mathcal T}$ be a decomposition of $T$ into UBVTs so that each vertex is in at most ${\Delta(T)/3}$ elements of $\widehat{\mathcal T}$.

First suppose that $\mathcal T\subseteq \widehat{\mathcal T}$.
Let $y$ be the vertex where Algorithm~\ref{alg:tree} halts.
If $y$ is the root of $T$, then $y$ is in ${\Delta(T)/3}$ trees in $\widehat{\mathcal T}$, and not all edges at $y$ are contained in those trees.
Hence $y$ is in at least ${\frac{\Delta(T)}{3}}+1$ elements of $\widehat{\mathcal T}$, a contradiction.
Now assume that $y$ is not the root of $T$.
After pruning, $y$ has degree at least $3$ or is still joined to a child by a red edge.
Therefore $y$ has descendants of degree $3$, and they do not all lie in a single path with $y$.
It follows that the tree consisting of $y$, its remaining descendants, and the edge joining $y$ to its parent is not a subgraph of a UBVT.
Therefore one $y$ or its descendants is in at least two elements of $\widehat{\mathcal T}$ that are not in $\mathcal T$.
Since $y$ and all of its remaining descendants are in ${\Delta(T)/3}-1$ elements of $\mathcal T$, some vertex is in at least ${\Delta(T)/3}+1$ elements of $\widehat{\mathcal T}$, a contradiction.

We now assume that $\mathcal T\not\subseteq \widehat{\mathcal T}$.
Choose $\widehat{\mathcal T}$ to maximize $k$ so that $T_i\in \widehat{\mathcal T}$ for all $i\le k$; note that $k<\ell$.
Let $v$ be the vertex where $T_{k+1}$ is pruned in Algorithm~\ref{alg:tree}.
Let $T'$ be the subtree of $T$ that remains when Algorithm~\ref{alg:tree} calls PRUNE at $v$.
Let $x_1,\ldots,x_m$ be the children of $v$, ordered so that when $i<j$ the index of the element of $\mathcal T$ containing $x_i$ is less than or equal to the index of the element of $\mathcal T$ containing $x_j$.

Let $T_{vx_i}'$ be the subtree of $T'$ consisting of $v$, $x_i$, and all descendants of $x_i$.
Let $\widehat T_{vx_i}$ be the tree in $\widehat{\mathcal T}$ that contains $vx_i$.
We claim that $T'_{vx_i}\subseteq \widehat T_{vx_i}$.
If $T'_{vx_i}\not \subseteq \widehat T_{vx_i}$, then there is a descendant $x'$ of $v$ in $\widehat T_{vx_i}$ such that some child of $x'$ from $T'$ is not in $\widehat T_{vx_i}$.
It follows that $x'$ lies in at least two elements of $\widehat{\mathcal T}-\{T_1,\ldots,T_k\}$.
Since $x'$ is a descendant of $v$, Algorithm~\ref{alg:tree} has already run PRUNE at $x'$ when it begins to run PRUNE at $v$.
Since $x'$ is not a leaf in $T'$, it follows that ${\Delta(T)/3}-1$ trees are pruned at $x'$ by Algorithm~\ref{alg:tree}.
Furthermore, the trees that are pruned at $x'$ are in the set $\{T_1,\ldots,T_k\}$ and hence are also in $\widehat{\mathcal T}$.
It follows that $x'$ lies in ${\Delta(T)/3}+1$ elements of $\widehat{\mathcal T}$, a contradiction.
Therefore, for all $i\in [m]$, $T'_{vx_i}$ is a subtree of an element of $\widehat{\mathcal T}$.

We now impose an additional extremal condition on $\widehat{\mathcal T}$.
Since $T_{k+1}$ is pruned at $v$, it follows that $T_{k+1}=T'_{vx_i}\cup\ldots\cup  T'_{vx_{i+j}}$ for some $i\in [m]$ and some $j\in \{0,1,2\}$.
Order $vx_i,\ldots, vx_{i+j}$ so that $vx_i$ has the highest priority color with respect to the order $R\succ Y\succ G$.
Among all optimal decompositions that contain $\{T_1,\ldots,T_k\}$, choose $\widehat{\mathcal T}$ so that $\widehat T_{vx_i}$, the element of $\widehat{\mathcal T}$ that contains $vx_i$, has as many branches in common with $T_{k+1}$ as possible.
We will modify $\widehat T_{vx_i}$ to obtain a new optimal decomposition of $T$ that contradicts the extremality of $\widehat{\mathcal T}$.
We proceed by cases depending on the relative color-types of $T_{k+1}$ and $\widehat T_{vx_i}$ (refer to Table~\ref{tab:prefCuts}).
The specific modifications to $\widehat T_{vx_i}$ are summarized in Table~\ref{table:Tk+1}.

\begin{table}
\centering
\begin{tabular}{cc}
&$c(\widehat T_{vx_i})$\\
&\\
$c(T_{k+1})$&
\begin{tabular}{r|cc}
	&RG			&R\\ \hline
RG	&Swap G-G	&Absorb G\\
R	&Give G		&Same 
\end{tabular}
\end{tabular}

\vspace{.35in}

\begin{tabular}{cc}
&$c(\widehat T_{vx_i})$\\
&\\
$c(T_{k+1})$&\begin{tabular}{r|ccccc}
		&YYG		&YY			&YGG		&YG			&Y\\ \hline
YYG		&Swap		&Absorb G	&Swap G-Y	&Absorb Y	&Absorb Y\\
YY		&Give G		&Swap Y-Y	&Swap G-Y	&Absorb Y	&Absorb Y\\
YGG		&Swap Y-G	&Absorb G	&Swap G-G	&Absorb G	&Absorb G\\
YG		&Give Y		&Absorb G	&Give G		&Swap G-G	&Absorb G\\
Y		&NA			&Give Y		&NA			&Give G		&Same\\
\end{tabular}
\end{tabular}

\vspace{.35in}

\begin{tabular}{cc}
&$c(\widehat T_{vx_i})$\\
&\\
$c(T_{k+1})$&
\begin{tabular}{r|cccccc}
	&RG			&YGG		&YG			&GGG		&GG			&G\\ \hline
GGG	&Swap R-GG	&Swap Y-G	&Absorb G	&Swap G-G	&Absorb G 	&Absorb G\\
GG 	&Swap R-G	&Give Y	&Swap Y-G	&Give G	&Swap G-G	&Absorb G\\
G 	&Give R	&NA			&Give Y	&NA			&Give G	&Same\\
\end{tabular}
\end{tabular}
\caption{Modification to $\widehat T_{vx_i}$.
The top table is when $vx_i$ is red, the middle is when $vx_i$ is yellow, and the bottom is when $vx_i$ is green.
``Give" indicates that a branch of $T_{vx_i}'$ either becomes its own tree or is absorbed by another tree.
``Same" indicates that $\widehat T_{vx_i}=T_{k+1}$.}\label{table:Tk+1}
\end{table}

{\bf Case 1:} {\it $c(T_{k+1})\succeq c(\widehat T_{vx_i})$.}
Because $T_{k+1}$ and $\widehat T_{vx_i}$ both contain the branch with $vx_i$, it follows that $d_{T_{k+1}}(v)\ge 2$ as otherwise $T_{k+1}=\widehat T_{vx_i}$.
In all such cases, by Lemma~\ref{lem:treeswap} $\widehat T_{vx_i}$ can either swap a branch with or absorb a branch from some element of $\widehat{\mathcal T}-\{T_1,\ldots,T_k\}$ to have more branches in common with $T_{k+1}$.
This violates the extremality of $\widehat{\mathcal T}$.

{\bf Case 2:} {\it $c(\widehat T_{vx_i})\succ c(T_{k+1})$, and $v$ has no children after $T_{k+1}$ is pruned.}
In this case, $\widehat T_{vx_i}$ must have a branch that contains the parent of $v$.
Because $T_{k+1}$ is pruned by Algorithm~\ref{alg:tree}, we conclude that at most ${\Delta(T)/3}-1$ trees are used by the algorithm for the edges joining $v$ to its children.
If $v$ is in at most ${\Delta(T)/3}-1$ trees in $\widehat{\mathcal T}$, the removing the branch from $\widehat T_{vx_i}$ that contains the parent of $v$ and using it as its own element yields an optimal decomposition.
In this decomposition, the color-type of the tree that contains $vx_i$ is lower ranked than $c(T_{k+1})$, yielding an instance of Case 1.
Otherwise $v$ is in ${\Delta(T)/3}$ trees in $\widehat{\mathcal T}$, two of which are in $\widehat{\mathcal T}-\{T_1,\ldots,T_k\}$.
It is straightforward to check that in all such cases Lemma~\ref{lem:treeswap} applies and we can perform branch swaps on $T_{vx_i}'$ and another element of $\widehat{\mathcal T}-\{T_1,\ldots,T_k\}$ that contains $v$ to obtain a decomposition that either contains $\{T_1,\ldots,T_{k+1}\}$ (violating the extremality of $\widehat{\mathcal T}$) or that is an instance of Case 1.

{\bf Case 3:} {\it $c(\widehat T_{vx_i})\succ c(T_{k+1})$, and $v$ has children after $T_{k+1}$ is pruned.}
In this case, $c(T_{k+1})\in\{\textrm{R},\textrm{YY},\textrm{YGG},\textrm{GGG}\}$.
In all of these cases, $\widehat T_{vx_i}$ has a branch $B$ that does not include $vx_i$ such that $c(T_{k+1})\succeq c(\widehat T_{vx_i}-B)$.
If $v$ is in less than ${\Delta(T)/3}$ elements of $\widehat{\mathcal T}$, then replacing $\widehat T_{vx_i}$ with $\widehat T_{vx_i}-B$ and $B$ yields an optimal decomposition of $T$ that is an instance of Case 1.
Thus we assume that $v$ is in ${\Delta(T)/3}$ elements of $\widehat{\mathcal T}$.

If $B$ is green, then $c(T_{k+1})\in\{\textrm{R},\textrm{YY}\}$.
By Algorithm~\ref{alg:tree}, all remaining branches at $v$ are red or yellow, and the green branch of $\widehat T_{vx_i}$ contains the parent of $v$.
Therefore all remaining elements of $\widehat{\mathcal T}$ that contain $v$ have color type R, YY, or Y, and by Lemma~\ref{lem:treeswap} any one of these can absorb $B$ from $\widehat T_{vx_i}$.

If $B$ is yellow, then $c(T_{k+1})\in\{\textrm{YGG},\textrm{GGG}\}$.
By Algorithm~\ref{alg:tree}, all remaining branches at $v$ are green.
In this case, by Lemma~\ref{lem:treeswap} it is possible to perform a branch-swap with $B$ and a green branch of a tree in $\widehat{\mathcal T}-\{T_1,\ldots,T_k\}$ rooted at $v$ with color type GGG, GG, or G.

If $B$ is red, then $c(T_{k+1})=\textrm{GGG}$ and $c(\widehat T_{vx_i})=\textrm{RG}$.
By Algorithm~\ref{alg:tree}, all remaining branches at $v$ are green.
Therefore all remaining elements of $\widehat{\mathcal T}$ containing $v$ have color type GGG, GG, or G.
Because these elements of $\widehat{\mathcal T}$ contain two branches of $T_{k+1}$, we can assume (using branch absorption if necessary) that one of the remaining elements of $\widehat{\mathcal T}$ that contains $v$ is of color-type GGG or GG.
Thus it is possible by Lemma~\ref{lem:treeswap} to swap $B$ for two green branches of a tree in $\widehat{\mathcal T}-\{T_1,\ldots,T_k\}$.
This completes Case 3.

We conclude that for each color of $B$ it is possible to obtain an optimal decomposition that is an instance of Case 1 or that contains $\{T_1,\ldots,T_{k+1}\}$, violating the extremality of $\widehat{\mathcal T}$.

It remains to show that UNIT\textunderscore BAR\textunderscore TREE$(T)$ runs in linear time.
At a vertex $v$, PRUNE runs in time $O(d(v))$ and COLOR runs in constant time.
Since $\sum_{V(T)}d(v)=|V(T)|-1$ and the post-order traversal takes $O(|V(T)|)$ time, it follows that UNIT\textunderscore BAR\textunderscore TREE$(T)$ runs in $O(|V(T)|)$ time.
\end{proof}

We note that UNIT\textunderscore BAR\textunderscore TREE$(T)$ only provides a decomposition into UBVTs when $ub(T)=\CL{\Delta(T)/3}$.
If $ub(T)=\CL{\frac{\Delta(T)+1}{3}}$, then modifying the PRUNE routine to allow up to $\CL{\Delta(T)/3}$ trees to be pruned at nonroot vertices and $\CL{\Delta(T)/3}+1$ trees to be pruned at the root will allow UNIT\textunderscore BAR\textunderscore TREE$(T)$ to produce an optimal decomposition.
This adds $O(|V(T)|)$ additional operations to the algorithm (one more prune at each vertex).
Thus an optimal decomposition of such a tree can also be found in linear time, resulting from possibly two iterations of UNIT\textunderscore BAR\textunderscore TREE$(T)$.
In~\cite{DV}, Dean and Veytsel provide a construction of a unit bar visibility representation of a UBVT that is clearly implementable in time linear in the number of vertices.
Since the sum of the sizes of the vertex sets in a decomposition of a tree $T$ into UBVTs is at most $2|V(T)|-2$, it follows that an optimal $t$-unit-bar visibility representation of a tree can be found in linear time.

In~\cite{DE-MHP}, Dean et al.\ introduced {\it unit rectangle visibility graphs}, where vertices are assigned to distinct axis-aligned unit rectangles in the plane and edges correspond to uninterrupted vertical or horizontal lines of sight between the rectangles.
They proved that a tree $T$ is a unit rectangle visibility graph if and only if $\Upsilon_{ub}(T)\le 2$.
In the conclusion of their paper, they ask if there is an efficient algorithm to determine if an arbitrary graph is a unit rectangle visibility graph, and note that the question is unanswered even for trees.
Algorithm~\ref{alg:tree} answers the question in the affirmative for trees.

\begin{corollary}
There is a linear time algorithm that determines if a tree is a unit rectangle visibility graph.
\end{corollary}

Given our results on trees, it is possible to bound the unit bar visibility number of planar graphs based on their maximum degree.

\begin{theorem}~\label{thm:planar}
If $G$ is a planar graph, then $ub(G)\le\CL{\frac{\Delta(G)+1}{3}}+2$.
\end{theorem}

\begin{proof}
Nash-Williams~\cite{NW} proved that every planar graph has arboricity at most $3$, so there is a decomposition of $G$ into trees such that every vertex is in three trees (note that some of these trees may be $K_1$).
Let $v\in V(G)$ and let $T_1$, $T_2$, and $T_3$ be the trees in the decomposition containing $v$.
By Theorem~\ref{thm:treebound}, for each $i\in[3]$ there is a decomposition of $T$ into UBVTs so that $v$ is in at most $\CL{\frac{d_{T_i}(v)+1}{3}}$ elements of the decomposition.
Therefore there is a decomposition of $G$ into UBVTs such that $v$ is contained in at most $\CL{\frac{d_{T_1}(v)+1}{3}}+\CL{\frac{d_{T_2}(v)+1}{3}}+\CL{\frac{d_{T_3}(v)+1}{3}}$ elements of the decomposition.
Since $d_G(v)=d_{T_1}(v)+d_{T_2}(v)+d_{T_3}(v)$, it follows that $\CL{\frac{d_{T_1}(v)+1}{3}}+\CL{\frac{d_{T_2}(v)+1}{3}}+\CL{\frac{d_{T_3}(v)+1}{3}}\le \CL{\frac{d_G(v)+1}{3}}+2$.
\end{proof}

The {\it girth} of a graph is the length of its longest cycle.
When a planar graph has large girth, we can slightly improve Theorem~\ref{thm:planar}.

\begin{theorem}~\label{thm:planargirth}
If $G$ is a planar graph with girth at least $7$, then $ub(G)\leq \CL{\frac{\Delta(G)+1}{3}}+1$ .
\end{theorem}
\begin{proof} 
Borodin~\cite{Borodin1976} proved that planar graphs with girth at least $7$ are acyclically 3-colorable.
That is, $V(G)$ can be partitioned into three sets such that the union of any two of the sets will induce a forest.
Each vertex will appear in exactly two of these forests.
Hence there is a decomposition of $G$ into trees such that each vertex is in at most two trees.
Let $v\in V(G)$ and let $T_1$ and $T_2$ be the trees that contain $v$.
By Theorem~\ref{thm:treebound}, there is a decomposition of $G$ into UBVTs so that $v$ is contained in at most $\CL{\frac{d_{T_1}(v)+1}{3}}+\CL{\frac{d_{T_2}(v)+1}{3}}$ elements of the decomposition.
Since $\CL{\frac{d_{T_1}(v)+1}{3}}+\CL{\frac{d_{T_2}(v)+1}{3}}\le \CL{\frac{d(v)+1}{3}}+1$, the result follows.
\end{proof}

\section{Complete Bipartite Graphs}\label{bipartite}

In~\cite{DV}, Dean and Veytsel characterized the complete bipartite graphs that are unit bar visibility graphs; unsurprisingly there are very few.

\begin{theorem}\label{ubvg bipartites}(Dean-Veytsel [2003]) The complete bipartite graph $K_{m,n}$, where $m \ge n$, is a UBVG if and only if $m \leq 3$ and $n=1$, or $m=n=2$.
\end{theorem}

In this section we study the unit bar visibility number of complete bipartite graphs.
Throughout this section, we let $K_{m,n}$ have partite sets $Y$ and $X$ with $Y=\{y_1,\ldots,y_m\}$ and $X=\{x_1,\ldots,x_n\}$.
We will also assume that $m\ge n\ge 2$.
Note that the case when $n=1$ is handled in Section~\ref{trees}.
We begin with two constructions of $t$-unit-bar visibility representations of $K_{m,n}$.

\begin{construction}\label{Knmconst}
The construction begins with $2\CL{\frac{m}{4}}+1$ horizontal line segments of length $n$, each of which will be subdivided into $n$ contiguous unit bars.
From bottom to top, the segments are assigned to $Y$ and $X$ in alternating fashion; thus both the top and bottom bars are assigned to $Y$.
For  $j\in [\CL{\frac m4}+1]$, let $Y_j$ be the $j$th segment assigned to $Y$ and for $j\in [\CL{\frac{m}{4}}]$ let $X_j$  be the $j$th segment assigned to $X$, indexing from bottom to top.
For $j\in [\CL{\frac m4}+1]$, the segment $Y_j$ has left endpoint $(j-1,j-1)$, and for $j\in [\CL{\frac{m}{4}}]$ the segment $X_j$ has left endpoint $(j-\frac{1}{2},j-\frac{1}{2})$.

In each segment, the unit bars from left to right will correspond to $n$ consecutive vertices in the ordering of $Y$ or $X$.
Thus providing the leftmost unit bar in each line segment determines the construction.
For $j\in [\CL{\frac m4}+1]$, the leftmost bar of the segment $Y_j$ is assigned to $y_{3j-2}$.
For $j\in [\CL{\frac m4}]$, the leftmost bar of the segment $X_j$ is assigned to $x_{2-j}$ (with the index taken modulo $n$).
See Figure~\ref{fig:Const1}.

After constructing the line segments, for each pair $(x_i,y_{i'})$ where no bar of $x_i$ sees a bar of $y_{i'}$, to the right of the line segments we add a unit bar for both $x_i$ and $y_{i'}$ that are visible to each other and no other bars.
\end{construction}

\begin{figure}
\centering
\begin{tikzpicture}
\draw [line width=2pt] (0,0)--(6.5,0);
\draw [line width=2pt, dotted] (6.5,0)--(7.5,0);
\draw [line width=2pt] (7.5,0)--(10,0);
\draw [line width=2pt] (0,-.2)--(0,.2);
\draw [line width=2pt] (1,-.2)--(1,.2);
\draw [line width=2pt] (2,-.2)--(2,.2);
\draw [line width=2pt] (3,-.2)--(3,.2);
\draw [line width=2pt] (4,-.2)--(4,.2);
\draw [line width=2pt] (5,-.2)--(5,.2);
\draw [line width=2pt] (6,-.2)--(6,.2);
\draw [line width=2pt] (8,-.2)--(8,.2);
\draw [line width=2pt] (9,-.2)--(9,.2);
\draw [line width=2pt] (10,-.2)--(10,.2);
\node [right] at (0,.2)  {$y_{1}$};
\node [right] at (1,.2) {$y_{2}$};
\node [right] at (2,.2) {$y_{3}$};
\node [right] at (3,.2) {$y_{4}$};
\node [right] at (4,.2) {$y_{5}$};
\node [right] at (5,.2) {$y_{6}$};
\node [right] at (6,.2) {$y_{7}$};
\node [right] at (8,.2) {$y_{n-1}$};
\node [right] at (9,.2) {$y_{n}$};
\node at (-1,0) {$Y_1$};

\draw [shift={(.5,.7)}, line width=2pt] (0,0)--(6.5,0);
\draw [shift={(.5,.7)}, line width=2pt, dotted] (6.5,0)--(7.5,0);
\draw [shift={(.5,.7)}, line width=2pt] (7.5,0)--(10,0);
\draw [shift={(.5,.7)}, line width=2pt] (0,-.2)--(0,.2);
\draw [shift={(.5,.7)}, line width=2pt] (1,-.2)--(1,.2);
\draw [shift={(.5,.7)}, line width=2pt] (2,-.2)--(2,.2);
\draw [shift={(.5,.7)}, line width=2pt] (3,-.2)--(3,.2);
\draw [shift={(.5,.7)}, line width=2pt] (4,-.2)--(4,.2);
\draw [shift={(.5,.7)}, line width=2pt] (5,-.2)--(5,.2);
\draw [shift={(.5,.7)}, line width=2pt] (6,-.2)--(6,.2);
\draw [shift={(.5,.7)}, line width=2pt] (8,-.2)--(8,.2);
\draw [shift={(.5,.7)}, line width=2pt] (9,-.2)--(9,.2);
\draw [shift={(.5,.7)}, line width=2pt] (10,-.2)--(10,.2);
\node [shift={(.5,.7)}, right] at (0,.2)  {$x_{1}$};
\node [shift={(.5,.7)}, right] at (1,.2)  {$x_{2}$};
\node [shift={(.5,.7)}, right] at (2,.2)  {$x_{3}$};
\node [shift={(.5,.7)}, right] at (3,.2)  {$x_{4}$};
\node [shift={(.5,.7)}, right] at (4,.2)  {$x_{5}$};
\node [shift={(.5,.7)}, right] at (5,.2)  {$x_{6}$};
\node [shift={(.5,.7)}, right] at (6,.2)  {$x_{7}$};
\node [shift={(.5,.7)}, right] at (8,.2)  {$x_{n-1}$};
\node [shift={(.5,.7)}, right] at (9,.2) {$x_{n}$};
\node [shift={(0, .7)}] at (-1,0) {$X_1$};

\draw [shift={(1,1.4)}, line width=2pt] (0,0)--(6.5,0);
\draw [shift={(1,1.4)}, line width=2pt, dotted] (6.5,0)--(7.5,0);
\draw [shift={(1,1.4)}, line width=2pt] (7.5,0)--(10,0);
\draw [shift={(1,1.4)}, line width=2pt] (0,-.2)--(0,.2);
\draw [shift={(1,1.4)}, line width=2pt] (1,-.2)--(1,.2);
\draw [shift={(1,1.4)}, line width=2pt] (2,-.2)--(2,.2);
\draw [shift={(1,1.4)}, line width=2pt] (3,-.2)--(3,.2);
\draw [shift={(1,1.4)}, line width=2pt] (4,-.2)--(4,.2);
\draw [shift={(1,1.4)}, line width=2pt] (5,-.2)--(5,.2);
\draw [shift={(1,1.4)}, line width=2pt] (6,-.2)--(6,.2);
\draw [shift={(1,1.4)}, line width=2pt] (8,-.2)--(8,.2);
\draw [shift={(1,1.4)}, line width=2pt] (9,-.2)--(9,.2);
\draw [shift={(1,1.4)}, line width=2pt] (10,-.2)--(10,.2);
\node [shift={(1,1.4)}, right] at (0,.2)  {$y_{4}$};
\node [shift={(1,1.4)}, right] at (1,.2)  {$y_{5}$};
\node [shift={(1,1.4)}, right] at (2,.2)  {$y_{6}$};
\node [shift={(1,1.4)}, right] at (3,.2)  {$y_{7}$};
\node [shift={(1,1.4)}, right] at (4,.2)  {$y_{8}$};
\node [shift={(1,1.4)}, right] at (5,.2)  {$y_{9}$};
\node [shift={(1,1.4)}, right] at (6,.2)  {$y_{10}$};
\node [shift={(1,1.4)}, right] at (8,.2)  {$y_{n+2}$};
\node [shift={(1,1.4)}, right] at (9,.2) {$y_{n+3}$};
\node [shift={(0, 1.4)}] at (-1,0) {$Y_2$};

\draw [shift={(1.5,2.1)}, line width=2pt] (0,0)--(6.5,0);
\draw [shift={(1.5,2.1)}, line width=2pt, dotted] (6.5,0)--(7.5,0);
\draw [shift={(1.5,2.1)}, line width=2pt] (7.5,0)--(10,0);
\draw [shift={(1.5,2.1)}, line width=2pt] (0,-.2)--(0,.2);
\draw [shift={(1.5,2.1)}, line width=2pt] (1,-.2)--(1,.2);
\draw [shift={(1.5,2.1)}, line width=2pt] (2,-.2)--(2,.2);
\draw [shift={(1.5,2.1)}, line width=2pt] (3,-.2)--(3,.2);
\draw [shift={(1.5,2.1)}, line width=2pt] (4,-.2)--(4,.2);
\draw [shift={(1.5,2.1)}, line width=2pt] (5,-.2)--(5,.2);
\draw [shift={(1.5,2.1)}, line width=2pt] (6,-.2)--(6,.2);
\draw [shift={(1.5,2.1)}, line width=2pt] (8,-.2)--(8,.2);
\draw [shift={(1.5,2.1)}, line width=2pt] (9,-.2)--(9,.2);
\draw [shift={(1.5,2.1)}, line width=2pt] (10,-.2)--(10,.2);
\node [shift={(1.5,2.1)}, right] at (0,.2)  {$x_{n}$};
\node [shift={(1.5,2.1)}, right] at (1,.2)  {$x_{1}$};
\node [shift={(1.5,2.1)}, right] at (2,.2)  {$x_{2}$};
\node [shift={(1.5,2.1)}, right] at (3,.2)  {$x_{3}$};
\node [shift={(1.5,2.1)}, right] at (4,.2)  {$x_{4}$};
\node [shift={(1.5,2.1)}, right] at (5,.2)  {$x_{5}$};
\node [shift={(1.5,2.1)}, right] at (6,.2)  {$x_{6}$};
\node [shift={(1.5,2.1)}, right] at (8,.2)  {$x_{n-2}$};
\node [shift={(1.5,2.1)}, right] at (9,.2) {$x_{n-1}$};
\node [shift={(0, 2.1)}] at (-1,0) {$X_2$};

\draw [shift={(2,2.8)}, line width=2pt] (0,0)--(6.5,0);
\draw [shift={(2,2.8)}, line width=2pt, dotted] (6.5,0)--(7.5,0);
\draw [shift={(2,2.8)}, line width=2pt] (7.5,0)--(10,0);
\draw [shift={(2,2.8)}, line width=2pt] (0,-.2)--(0,.2);
\draw [shift={(2,2.8)}, line width=2pt] (1,-.2)--(1,.2);
\draw [shift={(2,2.8)}, line width=2pt] (2,-.2)--(2,.2);
\draw [shift={(2,2.8)}, line width=2pt] (3,-.2)--(3,.2);
\draw [shift={(2,2.8)}, line width=2pt] (4,-.2)--(4,.2);
\draw [shift={(2,2.8)}, line width=2pt] (5,-.2)--(5,.2);
\draw [shift={(2,2.8)}, line width=2pt] (6,-.2)--(6,.2);
\draw [shift={(2,2.8)}, line width=2pt] (8,-.2)--(8,.2);
\draw [shift={(2,2.8)}, line width=2pt] (9,-.2)--(9,.2);
\draw [shift={(2,2.8)}, line width=2pt] (10,-.2)--(10,.2);
\node [shift={(2,2.8)}, right] at (0,.2)  {$y_{7}$};
\node [shift={(2,2.8)}, right] at (1,.2)  {$y_{8}$};
\node [shift={(2,2.8)}, right] at (2,.2)  {$y_{9}$};
\node [shift={(2,2.8)}, right] at (3,.2)  {$y_{10}$};
\node [shift={(2,2.8)}, right] at (4,.2)  {$y_{11}$};
\node [shift={(2,2.8)}, right] at (5,.2)  {$y_{12}$};
\node [shift={(2,2.8)}, right] at (6,.2)  {$y_{13}$};
\node [shift={(2,2.8)}, right] at (8,.2)  {$y_{n+5}$};
\node [shift={(2,2.8)}, right] at (9,.2) {$y_{n+6}$};
\node [shift={(0, 2.8)}] at (-1,0) {$Y_3$};

\draw [shift={(2.5,3.5)}, line width=2pt] (0,0)--(6.5,0);
\draw [shift={(2.5,3.5)}, line width=2pt, dotted] (6.5,0)--(7.5,0);
\draw [shift={(2.5,3.5)}, line width=2pt] (7.5,0)--(10,0);
\draw [shift={(2.5,3.5)}, line width=2pt] (0,-.2)--(0,.2);
\draw [shift={(2.5,3.5)}, line width=2pt] (1,-.2)--(1,.2);
\draw [shift={(2.5,3.5)}, line width=2pt] (2,-.2)--(2,.2);
\draw [shift={(2.5,3.5)}, line width=2pt] (3,-.2)--(3,.2);
\draw [shift={(2.5,3.5)}, line width=2pt] (4,-.2)--(4,.2);
\draw [shift={(2.5,3.5)}, line width=2pt] (5,-.2)--(5,.2);
\draw [shift={(2.5,3.5)}, line width=2pt] (6,-.2)--(6,.2);
\draw [shift={(2.5,3.5)}, line width=2pt] (8,-.2)--(8,.2);
\draw [shift={(2.5,3.5)}, line width=2pt] (9,-.2)--(9,.2);
\draw [shift={(2.5,3.5)}, line width=2pt] (10,-.2)--(10,.2);
\node [shift={(2.5,3.5)}, right] at (0,.2)  {$x_{n-1}$};
\node [shift={(2.5,3.5)}, right] at (1,.2)  {$x_{n}$};
\node [shift={(2.5,3.5)}, right] at (2,.2)  {$x_{1}$};
\node [shift={(2.5,3.5)}, right] at (3,.2)  {$x_{2}$};
\node [shift={(2.5,3.5)}, right] at (4,.2)  {$x_{3}$};
\node [shift={(2.5,3.5)}, right] at (5,.2)  {$x_{4}$};
\node [shift={(2.5,3.5)}, right] at (6,.2)  {$x_{5}$};
\node [shift={(2.5,3.5)}, right] at (8,.2)  {$x_{n-3}$};
\node [shift={(2.5,3.5)}, right] at (9,.2) {$x_{n-2}$};
\node [shift={(0, 3.5)}] at (-1,0) {$X_3$};

\node at (-1,4.7) {$\vdots$};
\draw [line width =2pt, dotted] (3,4)--(3.9,5.26);
\draw [line width =2pt, dotted] (12,4)--(12.9,5.26);

\draw [shift={(4,5.6)}, line width=2pt] (0,0)--(6.5,0);
\draw [shift={(4,5.6)}, line width=2pt, dotted] (6.5,0)--(7.5,0);
\draw [shift={(4,5.6)}, line width=2pt] (7.5,0)--(10,0);
\draw [shift={(4,5.6)}, line width=2pt] (0,-.2)--(0,.2);
\draw [shift={(4,5.6)}, line width=2pt] (1,-.2)--(1,.2);
\draw [shift={(4,5.6)}, line width=2pt] (2,-.2)--(2,.2);
\draw [shift={(4,5.6)}, line width=2pt] (3,-.2)--(3,.2);
\draw [shift={(4,5.6)}, line width=2pt] (4,-.2)--(4,.2);
\draw [shift={(4,5.6)}, line width=2pt] (5,-.2)--(5,.2);
\draw [shift={(4,5.6)}, line width=2pt] (6,-.2)--(6,.2);
\draw [shift={(4,5.6)}, line width=2pt] (8,-.2)--(8,.2);
\draw [shift={(4,5.6)}, line width=2pt] (9,-.2)--(9,.2);
\draw [shift={(4,5.6)}, line width=2pt] (10,-.2)--(10,.2);
\node [shift={(4,5.6)}] at (-1,0)  {$x_{n-\FL{\frac{m}{4}}+2}$};
\node [shift={(4,5.6)}, right] at (9,-.45) {$x_{\FL{\frac{m}{4}}+2}$};
\node [shift={(0, 5.6)}] at (-1,0) {$X_{\CL{m/4}}$};

\draw [shift={(4.5,6.3)}, line width=2pt] (0,0)--(6.5,0);
\draw [shift={(4.5,6.3)}, line width=2pt, dotted] (6.5,0)--(7.5,0);
\draw [shift={(4.5,6.3)}, line width=2pt] (7.5,0)--(10,0);
\draw [shift={(4.5,6.3)}, line width=2pt] (0,-.2)--(0,.2);
\draw [shift={(4.5,6.3)}, line width=2pt] (1,-.2)--(1,.2);
\draw [shift={(4.5,6.3)}, line width=2pt] (2,-.2)--(2,.2);
\draw [shift={(4.5,6.3)}, line width=2pt] (3,-.2)--(3,.2);
\draw [shift={(4.5,6.3)}, line width=2pt] (4,-.2)--(4,.2);
\draw [shift={(4.5,6.3)}, line width=2pt] (5,-.2)--(5,.2);
\draw [shift={(4.5,6.3)}, line width=2pt] (6,-.2)--(6,.2);
\draw [shift={(4.5,6.3)}, line width=2pt] (8,-.2)--(8,.2);
\draw [shift={(4.5,6.3)}, line width=2pt] (9,-.2)--(9,.2);
\draw [shift={(4.5,6.3)}, line width=2pt] (10,-.2)--(10,.2);
\node [shift={(4.5,6.3)},] at (-1,0)  {$y_{3\CL{\frac{m}{4}}+1}$};
\node [shift={(4.5,6.3)}, right] at (9,.45) {$y_{n+3\CL{\frac{m}{4}}+1}$};
\node [shift={(0, 6.3)}] at (-1,0) {$Y_{\CL{m/4}+1}$};

\end{tikzpicture}

\caption{The line segments in Construction~\ref{Knmconst}.}\label{fig:Const1}
\end{figure}
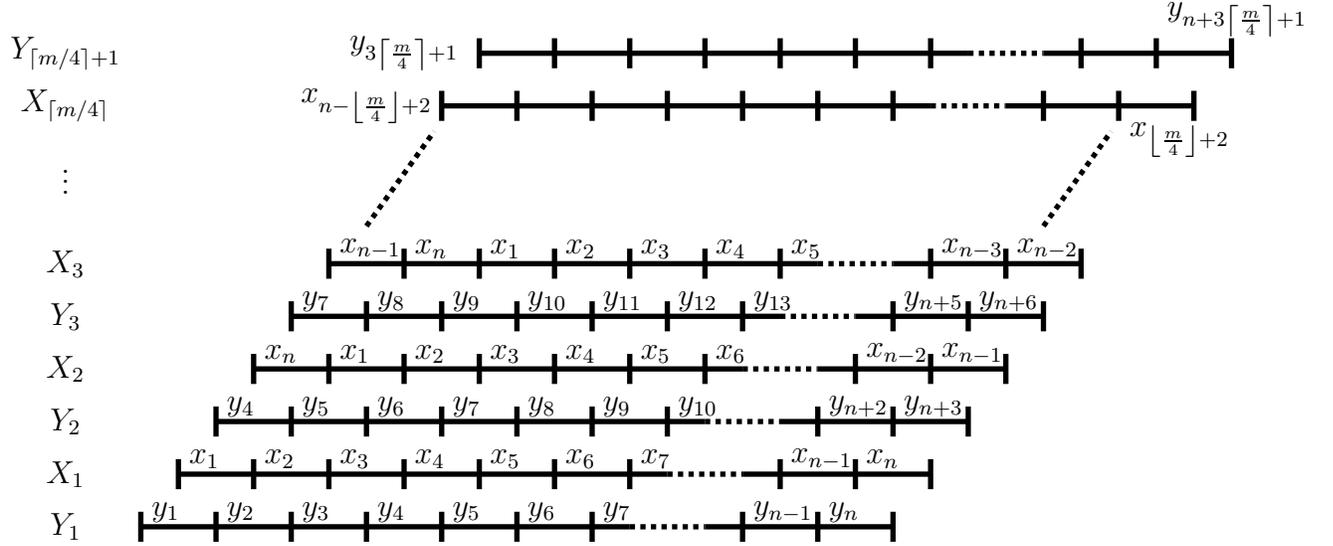

\begin{lemma}\label{lem:KmnUBcm}
The unit bar visibility layout of Construction~\ref{Knmconst} is a $t$-unit-bar visibiliy representation of $K_{m,n}$ with $t\le\CL{\frac m4}+\frac{2m}{n}+12$.
\end{lemma}

\begin{proof}
For $i\in [n]$ and $j\in [\CL{m/4}]$, let $b_{i,j}$ be the bar in $X_j$ that is assigned to $x_i$.
For $i\in [m]$ and $j\in [\CL{m/4}+1]$, let $\beta_{i,j}$ be the bar in $Y_j$ that is assigned to $y_i$ (note that $\beta_{i,j}$ may not exist for certain combinations of $i$ and $j$).
Fix $j\in [\CL{m/4}]$ and $i\in[n]$.
If $i\not \equiv 2-j,1-j\pmod n$, then $b_{i,j}$ sees bars $\beta_{i+4j-4,j}$, $\beta_{i+4j-3,j}$, $\beta_{i+4j-2,j+1}$, and $\beta_{i+4j-1,j+1}$.
If $i\equiv 2-j\pmod n$, then $b_{i,j}$ is the leftmost bar of $X_j$ and sees bars $\beta_{i+4j-4,j}$, $\beta_{i+4j-3,j}$, and $\beta_{i+4j-1,j+1}$.
If $i\equiv 1-j\pmod n$, then $b_{i,j}$ is the rightmost bar of $X_j$ and sees bars $\beta_{i+4j-4,j}$, $\beta_{i+4j-2,j+1}$, and $\beta_{i+4j-1,j+1}$.
Therefore, given $i\in [n]$, the sets of bars seen by $b_{i,1},\ldots,b_{i,\CL{m/4}}$ correspond to pairwise disjoint sets of vertices in $Y$.
It also follows that given $i\in [m]$, the bars in $Y_1,\ldots,Y_{\CL{m/4}+1}$ assigned to $y_i$ see bars corresponding to pairwise disjoint sets of vertices in $X$.

Let $x_i\in X$.
Observe that $x_i$ is assigned to the leftmost bar of at most $\CL{\frac {\CL{m/4}}{n}}$ segments.
Similarly, $x_i$ is assigned to the rightmost bar of at most $\CL{\frac {\CL{m/4}}{n}}$ segments.
Interior bars in the segments see four other bars, and leftmost and rightmost bars see three other bars, so the bars assigned to $x_i$ in $X_1\cup\ldots\cup X_{\CL{m/4}}$ see at least $4\left(\CL{\frac{m}{4}}-2\CL{\frac {\CL{m/4}}{n}}\right)$ bars, which are assigned to distinct vertices in $Y$.
Thus $x_i$ is assigned to at most $\frac{2m}{n}+12$ additional bars.
Therefore the vertices in $X$ are assigned to at most $\CL{\frac m4}+\frac{2m}{n}+12$ bars.

Each interior bar of $Y_j$ for $2\le j\le \CL{m/4}$ sees four other bars, while the leftmost and rightmost bars see three other bars.
The leftmost bar of $Y_1$ and the rightmost bar of $Y_{\CL{\frac m4}+1}$ each see one other bar.
All other bars in $Y_1$ and $Y_{\CL{\frac m4}+1}$ see two other bars.
Let $y_i\in Y$.
Observe that $y_i$ is assigned to the leftmost bar of at most two segments and the rightmost bar of at most two segments.
Furthermore, there are at most $\CL{(m-n)/3}$ segments that do not contain a bar assigned to $y_i$.
Suppose that $y_i$ is not assigned to a bar in $r$ of the segments.
Since $y_i$ is assigned to at most one bar in $Y_1$ and $Y_{\CL{m/4}+1}$, it follows that the bars of $y_i$ in the segments see at least $4(\CL{\frac m4}+1-r)-8$ bars, which are assigned to distinct vertices in $X$.
Thus $y_i$ needs at most $n-(4(\CL{\frac m4}+1-r)-8)$ additional bars.
Since
\begin{align*}
\CL{\frac m4}+1-r+n-\left(4\left(\CL{\frac m4}+1-r\right)-8\right)&\le \CL{\frac m4}+3r+n-m+5\\
&\le\CL{\frac m4}+3\CL{\frac{m-n}{3}}+n-m+5\\
&\le\CL{\frac m4}+8,
\end{align*}
it follows that $y_j$ is assigned to at most $\CL{\frac m4}+8$ bars.

\end{proof}

\begin{construction}\label{const:Kmo(m)}
We begin with an array of $2\FL{n/4}+1$ line segments of length $n\FL{m/n}$.
From bottom to top, the segments are assigned to $Y$ and $X$ in alternating fashion.
For  $j\in [\FL{n/4}+1]$, let $Y_j$ for be the $j$th segment assigned to $Y$ and for $j\in [\FL{n/4}]$ let $X_j$  be the $j$th segment assigned to $X$, indexing from bottom to top.
For $j\in [\FL{n/4}+1]$, the segment $Y_j$ has left endpoint $(j-1,j-1)$, and for $j\in [\FL{n/4}]$ the segment $X_j$ has left endpoint $(j-\frac{1}{2},j-\frac{1}{2})$.
The segments assigned to $X$ consist of $\FL{m/n}$ sets of bars assigned to the vertices of $X$ in order, up to a cyclic shift.
The segments assigned to $Y$ consist of $n\FL{m/n}$ bars assigned to $\{y_1,\ldots,y_{n\FL{m/n}}\}$ up to a cyclic shift.
For $j\in [\FL{\frac n4}+1]$, the leftmost bar of the segment $Y_j$ is assigned to $y_{3j-2}$.
For $j\in [\FL{\frac n4}]$, the leftmost bar of the segment $X_j$ is assigned to $x_{2-i}$.
See Figure~\ref{fig:Kmo(m0}.

After constructing the line segments, for each pair $(x_i,y_{i'})$ where no bar of $x_i$ sees a bar of $y_{i'}$,  to the right of the line segments we add a unit bar for both $x_i$ and $y_{i'}$ that are visible to each other and no other bars.
\end{construction}

\begin{figure}
\begin{tikzpicture}
\draw [line width=2pt] (0,0) -- (9,0);
\draw [line width=2pt] (0,-.2)--(0,.2);
\draw [line width=2pt] (9,-.2)--(9,.2);
\draw [line width=1pt] (1,-.1)--(1,.1);
\draw [line width=1pt] (2,-.1)--(2,.1);
\draw [line width=1pt] (3,-.1)--(3,.1);
\draw [line width=1pt] (4,-.1)--(4,.1);
\draw [line width=1pt] (5,-.1)--(5,.1);
\draw [line width=1pt] (6,-.1)--(6,.1);
\draw [line width=1pt] (7,-.1)--(7,.1);
\draw [line width=1pt] (8,-.1)--(8,.1);
\node [right] at (0,.3)  {$y_{1}$};
\node [right] at (1,.3)  {$y_{2}$};
\node at (8.5,.3)  {$y_{n\FL{\frac mn}}$};
\node at (-1,0) {$Y_1$};


\draw [line width=2pt] (0.5,.8) -- (9.5,.8);
\draw [line width=2pt] (0.5,.6)--(0.5,1);
\draw [line width=2pt] (3.5,.6)--(3.5,1);
\draw [line width=2pt] (6.5,.6)--(6.5,1);
\draw [line width=2pt] (9.5,.6)--(9.5,1);
\draw [line width=1pt] (1.5,.7)--(1.5,.9);
\draw [line width=1pt] (2.5,.7)--(2.5,.9);
\draw [line width=1pt] (4.5,.7)--(4.5,.9);
\draw [line width=1pt] (5.5,.7)--(5.5,.9);
\draw [line width=1pt] (7.5,.7)--(7.5,.9);
\draw [line width=1pt] (8.5,.7)--(8.5,.9);
\node [right] at (0.5,1.1)  {$x_{1}$};
\node [right] at (3.5,1.1)  {$x_{1}$};
\node [right] at (6.5,1.1)  {$x_{1}$};
\node [right] at (2.5,1.1)  {$x_{n}$};
\node [right] at (5.5,1.1)  {$x_{n}$};
\node [right] at (8.5,1.1)  {$x_{n}$};
\node at (-1,.8) {$X_1$};


\draw [shift={(1,1.6)}, line width=2pt] (0,0) -- (9,0);
\draw [shift={(1,1.6)}, line width=2pt] (0,-.2)--(0,.2);
\draw [shift={(1,1.6)}, line width=2pt] (9,-.2)--(9,.2);
\draw [shift={(1,1.6)}, line width=1pt] (1,-.1)--(1,.1);
\draw [shift={(1,1.6)}, line width=1pt] (2,-.1)--(2,.1);
\draw [shift={(1,1.6)}, line width=1pt] (3,-.1)--(3,.1);
\draw [shift={(1,1.6)}, line width=1pt] (4,-.1)--(4,.1);
\draw [shift={(1,1.6)}, line width=1pt] (5,-.1)--(5,.1);
\draw [shift={(1,1.6)}, line width=1pt] (6,-.1)--(6,.1);
\draw [shift={(1,1.6)}, line width=1pt] (7,-.1)--(7,.1);
\draw [shift={(1,1.6)}, line width=1pt] (8,-.1)--(8,.1);
\node [shift={(1,1.6)}, right] at (0,.3)  {$y_{4}$};
\node [shift={(1,1.6)}, right] at (1,.3)  {$y_{5}$};
\node [shift={(1,1.6)}] at (8.5,.3)  {$y_{3}$};
\node [shift={(0,1.6)}] at (-1,0) {$Y_2$};


\draw [shift={(1,1.6)}, line width=2pt] (0.5,.8) -- (9.5,.8);
\draw [shift={(1,1.6)}, line width=2pt] (0.5,.6)--(0.5,1);
\draw [shift={(1,1.6)}, line width=2pt] (3.5,.6)--(3.5,1);
\draw [shift={(1,1.6)}, line width=2pt] (6.5,.6)--(6.5,1);
\draw [shift={(1,1.6)}, line width=2pt] (9.5,.6)--(9.5,1);
\draw [shift={(1,1.6)}, line width=1pt] (1.5,.7)--(1.5,.9);
\draw [shift={(1,1.6)}, line width=1pt] (2.5,.7)--(2.5,.9);
\draw [shift={(1,1.6)}, line width=1pt] (4.5,.7)--(4.5,.9);
\draw [shift={(1,1.6)}, line width=1pt] (5.5,.7)--(5.5,.9);
\draw [shift={(1,1.6)}, line width=1pt] (7.5,.7)--(7.5,.9);
\draw [shift={(1,1.6)}, line width=1pt] (8.5,.7)--(8.5,.9);
\node [shift={(1,1.6)}, right] at (0.5,1.1)  {$x_{n}$};
\node [shift={(1,1.6)}, right] at (3.5,1.1)  {$x_{n}$};
\node [shift={(1,1.6)}, right] at (6.5,1.1)  {$x_{n}$};
\node [shift={(1,1.6)}, right] at (2.5,1.1)  {$x_{n-1}$};
\node [shift={(1,1.6)}, right] at (5.5,1.1)  {$x_{n-1}$};
\node [shift={(1,1.6)}, right] at (8.5,1.1)  {$x_{n-1}$};
\node [shift={(0,1.6)}] at (-1,.8) {$X_2$};


\draw [shift={(2,3.2)}, line width=2pt] (0,0) -- (9,0);
\draw [shift={(2,3.2)}, line width=2pt] (0,-.2)--(0,.2);
\draw [shift={(2,3.2)}, line width=2pt] (9,-.2)--(9,.2);
\draw [shift={(2,3.2)}, line width=1pt] (1,-.1)--(1,.1);
\draw [shift={(2,3.2)}, line width=1pt] (2,-.1)--(2,.1);
\draw [shift={(2,3.2)}, line width=1pt] (3,-.1)--(3,.1);
\draw [shift={(2,3.2)}, line width=1pt] (4,-.1)--(4,.1);
\draw [shift={(2,3.2)}, line width=1pt] (5,-.1)--(5,.1);
\draw [shift={(2,3.2)}, line width=1pt] (6,-.1)--(6,.1);
\draw [shift={(2,3.2)}, line width=1pt] (7,-.1)--(7,.1);
\draw [shift={(2,3.2)}, line width=1pt] (8,-.1)--(8,.1);
\node [shift={(2,3.2)}, right] at (0,.3)  {$y_{7}$};
\node [shift={(2,3.2)}, right] at (1,.3)  {$y_{8}$};
\node [shift={(2,3.2)}] at (8.5,.3)  {$y_{6}$};
\node [shift={(0,3.2)}] at (-1,0) {$Y_3$};


\draw [shift={(2,3.2)}, line width=2pt] (0.5,.8) -- (9.5,.8);
\draw [shift={(2,3.2)}, line width=2pt] (0.5,.6)--(0.5,1);
\draw [shift={(2,3.2)}, line width=2pt] (3.5,.6)--(3.5,1);
\draw [shift={(2,3.2)}, line width=2pt] (6.5,.6)--(6.5,1);
\draw [shift={(2,3.2)}, line width=2pt] (9.5,.6)--(9.5,1);
\draw [shift={(2,3.2)}, line width=1pt] (1.5,.7)--(1.5,.9);
\draw [shift={(2,3.2)}, line width=1pt] (2.5,.7)--(2.5,.9);
\draw [shift={(2,3.2)}, line width=1pt] (4.5,.7)--(4.5,.9);
\draw [shift={(2,3.2)}, line width=1pt] (5.5,.7)--(5.5,.9);
\draw [shift={(2,3.2)}, line width=1pt] (7.5,.7)--(7.5,.9);
\draw [shift={(2,3.2)}, line width=1pt] (8.5,.7)--(8.5,.9);
\node [shift={(2,3.2)}, right] at (0.5,1.1)  {$x_{n-1}$};
\node [shift={(2,3.2)}, right] at (3.5,1.1)  {$x_{n-1}$};
\node [shift={(2,3.2)}, right] at (6.5,1.1)  {$x_{n-1}$};
\node [shift={(2,3.2)}, right] at (2.5,1.1)  {$x_{n-2}$};
\node [shift={(2,3.2)}, right] at (5.5,1.1)  {$x_{n-2}$};
\node [shift={(2,3.2)}, right] at (8.5,1.1)  {$x_{n-2}$};
\node [shift={(0,3.2)}] at (-1,.8) {$X_3$};


\draw [shift={(3.5,5.6)}, line width=2pt] (0.5,.8) -- (9.5,.8);
\draw [shift={(3.5,5.6)}, line width=2pt] (0.5,.6)--(0.5,1);
\draw [shift={(3.5,5.6)}, line width=2pt] (3.5,.6)--(3.5,1);
\draw [shift={(3.5,5.6)}, line width=2pt] (6.5,.6)--(6.5,1);
\draw [shift={(3.5,5.6)}, line width=2pt] (9.5,.6)--(9.5,1);
\draw [shift={(3.5,5.6)}, line width=1pt] (1.5,.7)--(1.5,.9);
\draw [shift={(3.5,5.6)}, line width=1pt] (2.5,.7)--(2.5,.9);
\draw [shift={(3.5,5.6)}, line width=1pt] (4.5,.7)--(4.5,.9);
\draw [shift={(3.5,5.6)}, line width=1pt] (5.5,.7)--(5.5,.9);
\draw [shift={(3.5,5.6)}, line width=1pt] (7.5,.7)--(7.5,.9);
\draw [shift={(3.5,5.6)}, line width=1pt] (8.5,.7)--(8.5,.9);
\node [shift={(3.5,5.6)}, right] at (0.5,.4)  {$x_{n-\FL{\frac n4}+2}$};
\node [shift={(3.5,5.6)}, right] at (8.5,.4)  {$x_{n-\FL{\frac n4}+1}$};
\node [shift={(0,5.6)}] at (-1,.8) {$X_{\FL{n/4}}$};


\draw [shift={(4.5,7.2)}, line width=2pt] (0,0) -- (9,0);
\draw [shift={(4.5,7.2)}, line width=2pt] (0,-.2)--(0,.2);
\draw [shift={(4.5,7.2)}, line width=2pt] (9,-.2)--(9,.2);
\draw [shift={(4.5,7.2)}, line width=1pt] (1,-.1)--(1,.1);
\draw [shift={(4.5,7.2)}, line width=1pt] (2,-.1)--(2,.1);
\draw [shift={(4.5,7.2)}, line width=1pt] (3,-.1)--(3,.1);
\draw [shift={(4.5,7.2)}, line width=1pt] (4,-.1)--(4,.1);
\draw [shift={(4.5,7.2)}, line width=1pt] (5,-.1)--(5,.1);
\draw [shift={(4.5,7.2)}, line width=1pt] (6,-.1)--(6,.1);
\draw [shift={(4.5,7.2)}, line width=1pt] (7,-.1)--(7,.1);
\draw [shift={(4.5,7.2)}, line width=1pt] (8,-.1)--(8,.1);
\node [shift={(4.5,7.2)}, right] at (0,.4)  {$y_{3\FL{\frac n4}+1}$};
\node [shift={(4.5,7.2)}] at (8.5,.4)  {$y_{3\FL{\frac n4}}$};
\node [shift={(0,7.2)}] at (-1,0) {$Y_{\FL{n/4}+1}$};

\node at (-1,5.2) {$\vdots$};
\draw [line width=2pt, dotted] (3,4.6)--(4,5.94);
\draw [line width=2pt, dotted] (11,4.6)--(12,5.94);

\end{tikzpicture}
\caption{The line segments in Construction~\ref{const:Kmo(m)}.}\label{fig:Kmo(m0}
\end{figure}
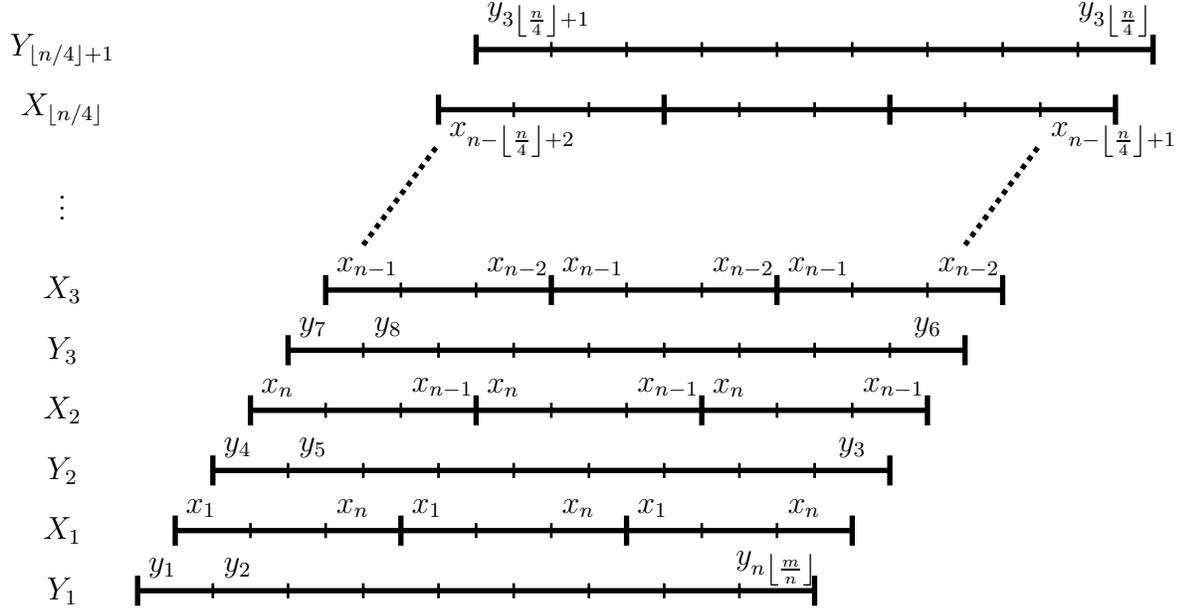

\begin{lemma}\label{lem:KmnUBo(m)}
The unit bar visibility layout of Construction~\ref{const:Kmo(m)} is a $t$-unit-bar visibiliy representation of $K_{m,n}$ with $t\le \frac m4+n+1$.
\end{lemma}

\begin{proof}
Fix $j\in [\FL{n/4}]$ and $i\in[n]$.
By construction, if a bar in $X_j$ that is assigned to $x_{i}$ sees a bar that is assigned to $y_{i'}$, then $i'\in [n\FL{m/n}]$ and $i'\equiv i+4j-4,i+4j-3,i+4j-2$, or $i+4j-1\pmod n$.
The leftmost and rightmost bars of $X_j$ see three other bars, and all other bars of $X_j$ see four other bars.
Furthermore, the bars in $X_j$ that are assigned to $x_{i}$ see pairwise disjoint sets of bars.

Each vertex in $X$ is assigned to the leftmost bar of at most one segment and the rightmost bar of at most one segment.
Therefore each vertex in $X$ is assigned at most two additional bars outside of the segments for visibility to vertices in $\{y_1,\ldots,y_{n\FL{m/n}}\}$.
Each vertex in $X$ is also assigned at most $n-1$ bars for visibility to vertices in $\{y_{n\FL{m/n}+1},\ldots,y_m\}$.
Thus each vertex in $X$ is assigned at most $\FL{\frac{n}{4}}\FL{\frac{m}{n}}+2+n-1\le \frac{m}{4}+n+1$ bars.

Each vertex in $Y$ is assigned at most $n$ additional bars outside of the segments.
Thus each vertex in $Y$ is assigned at most $\FL{n/4}+1+n$ bars.
Since $m\ge n$, it follows that each vertex in $Y$ is assigned at most $\frac m4+1+n$ bars.
\end{proof}

\begin{theorem}\label{thm:KmnUB}
For positive integers $m$ and $n$ with $m\ge n\ge 2$, 
$$ub(K_{m,n})\le \min\left\{\CL{\frac m4}+\frac{2m}{n}+12,\frac m4+n+1\right\}.$$
Therefore $ub(K_{m,n})\le \frac{m}{4}+o(m)$.
\end{theorem}

\begin{proof}
The first statement follows immediately from Lemmas~\ref{lem:KmnUBcm} and~\ref{lem:KmnUBo(m)}.
For the second statement we use $ub(K_{m,n})\le \min\left\{\frac m4+\frac{2m}{n}+13,\frac m4+n+1\right\}$.
If $n\le 6+\sqrt{36+2m}$, then $ub(K_{m,n})\le \frac m4+(6+\sqrt{36+2m})+1$.
If $n\ge 6+\sqrt{36+2m}$, then $ub(K_{m,n})\le \frac m4+\frac{2m}{6+\sqrt{36+2m}}+12$.
In both cases, $ub(K_{m,n})\le \frac m4+o(m)$.
\end{proof}

We next prove a general lower bound on the unit bar visibility number of $K_{m,n}$.
We refer to the bars in a $t$-unit-bar visibility representation of $K_{m,n}$ that correspond to vertices in $Y$ and $X$ as {\it $Y$-bars} and {\it $X$-bars}, respectively.

\begin{theorem}~\label{thm:bipartitem/5}
For positive integers $m$ and $n$ with $m\ge n\ge 2$, $ub(K_{m,n})\ge \CL{\frac m5}$.
\end{theorem}

\begin{proof}
Again, we let $K_{m,n}$ have vertex set $X\cup Y$, where $|X|=n$ and $|Y|=m$.
Let $R$ be a $t$-unit-bar visibility representation of $K_{m,n}$.
We will assign the edges in $K_{m,n}$ to the $X$-bars in $R$.
For each edge $xy$ in $K_{m,n}$, pick a line of sight in $R$ between bars $b_x$ and $b_y$ for $x$ and $y$, respectively.
Each line of sight defines a maximal axis-aligned rectangle whose interior has empty intersection with all bars; we select the left edge of the rectangle and call this line segment $\ell_{xy}$.
If $\ell_{xy}$ contains an interior point of $b_x$, we can refer to the line of sight to the \textit{left} of $\ell_{xy}$: this is the line of sight from a point arbitrarily close to the intersection of $b_x$ and $\ell_{xy}$ on the left that runs in the same direction from $b_x$ as $\ell_{xy}$.
Note that the line of sight to the left of $\ell_{xy}$ may not be used the in assignment of edges to lines of sight, or may not connect $b_x$ to another bar.
We say that two bars are {\it aligned} if their left endpoint has the same $x$-coordinate (that is, the bars have the same projection onto the $x$-axis).
We assign the edge $xy$ to a bar in $R$ as follows (see Figure~\ref{fig:edgeassign}):

\begin{center}
\begin{tabular}{rl}
{\bf Type 1:}& $\ell_{xy}$ contains the left endpoint of $b_x$. 
Assign $xy$ to $b_x$.\\
&\\
{\bf Type 2:}& 
\begin{minipage}{5in}
$\ell_{xy}$ contains an interior point of $b_x$ and the left endpoint of $b_y$, and the 
\end{minipage}\\

& 
\begin{minipage}{5in}
line of sight to the left of $\ell_{xy}$ either sees no other bar or sees a bar $b_{y'}$ where $b_{y'}$ does not overlap $b_y$.
Assign $xy$ to $b_x$.
\end{minipage}\\
&\\

{\bf Type 3:}& 
\begin{minipage}{5in}
$\ell_{xy}$ contains an interior point of $b_x$ and the left endpoint of $b_y$, and the 
\end{minipage}\\
&\begin{minipage}{5in}
line of sight to the left of $\ell_{xy}$ connects $b_x$ to a bar $b_{y'}$ where $b_{y'}$ overlaps $b_y$.
In this case, there is a bar $b_{x'}$ assigned to a vertex $x'\in X$ such $b_{x'}$ and $b_y$ are aligned, and the continuation of $\ell_{xy}$ through $b_y$ is the left edge of a line of sight between $b_y$ and $b_{x'}$.
Assign $xy$ to $b_{x'}$.\end{minipage}\\
&\\

{\bf Type 4:}& 
\begin{minipage}{5in}
$\ell_{xy}$ contains an interior point of both $b_x$ and $b_y$.
It follows that there \end{minipage}\\
&
\begin{minipage}{5in}are bars $b_{y'}$ and $b_{x'}$ assigned to vertices $y'\in Y$ and $x'\in X$ so that $b_x$ can see $b_{y'}$, $b_{x'}$ and $b_{y'}$ are aligned, $\ell_{xy}$ contains the right endpoints of $b_{y'}$ and $b_{x'}$, and $\ell_{xy}$ is the right edge of a line of sight between $b_{y'}$ and $b_{x'}$.
Assign $xy$ to $b_{x'}$.\end{minipage}\\

\end{tabular}
\end{center}

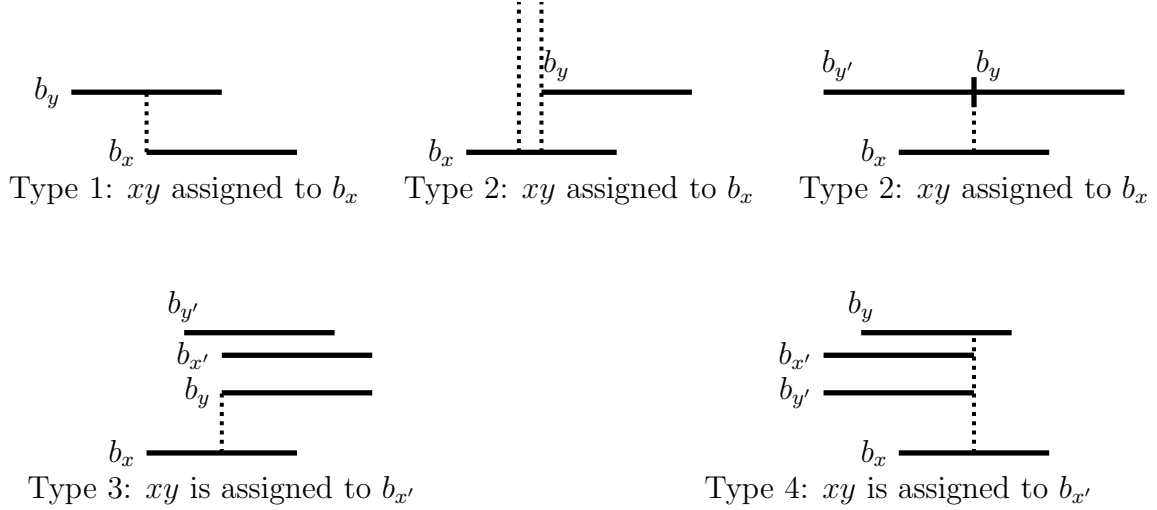
\begin{figure}
\centering
\begin{tikzpicture}
\draw [line width=2pt] (0,0)--(2,0);
\node [left] at (0,0) {$b_x$};
\draw [line width=2pt] (-1,.8)--(1,.8);
\node [left] at (-1,.8) {$b_y$};
\draw [dotted, line width =1.5pt] (0,0)--(0,.8);
\node at (.5,-.5) {Type 1: $xy$ assigned to $b_x$};

\draw [line width=2pt] (4.25,0)--(6.25,0);
\node [left] at (4.25,0) {$b_x$};
\draw [line width=2pt] (5.25,.8)--(7.25,.8);
\node [above] at (5.45,.8) {$b_y$};
\draw [dotted, line width =1.5pt] (5.25,0)--(5.25,2);
\draw [dotted, line width =1.5pt] (4.95,0)--(4.95,2);
\node at (5.75,-.5) {Type 2: $xy$ assigned to $b_x$};

\draw [shift={(6,0)},line width=2pt] (4,0)--(6,0);
\node [shift={(6,0)},left] at (4,0) {$b_x$};
\draw [shift={(6,0)},line width=2pt] (5,.8)--(7,.8);
\node [shift={(6,0)},above] at (5.2,.8) {$b_y$};
\draw [shift={(6,0)},dotted, line width =1.5pt] (5,0)--(5,.8);
\draw [shift={(6,0)}, line width =2pt] (3,.8)--(5,.8);
\node [shift={(6,0)},above] at (3.2,.8) {$b_{y'}$};
\draw [shift={(6,0)}, line width =2pt] (5,.6)--(5,1);
\node [shift={(6,0)}] at (5,-.5) {Type 2: $xy$ assigned to $b_x$};

\draw [shift={(-4,-4)},line width=2pt] (4,0)--(6,0);
\node [shift={(-4,-4)},left] at (4,0) {$b_x$};
\draw [shift={(-4,-4)},line width=2pt] (5,.8)--(7,.8);
\node [shift={(-4,-4)},left] at (5,.8) {$b_y$};
\draw [shift={(-4,-4)},dotted, line width =1.5pt] (5,0)--(5,.8);
\draw [shift={(-4,-4)}, line width =2pt] (4.5,1.6)--(6.5,1.6);
\node [shift={(-4,-4)},above] at (4.5,1.6) {$b_{y'}$};
\draw [shift={(-4,-4)}, line width =2pt] (5,1.3)--(7,1.3);
\node [shift={(-4,-4)},left] at (5,1.3) {$b_{x'}$};
\node [shift={(-4,-4)}] at (5,-.5) {Type 3: $xy$ is assigned to $b_{x'}$};

\draw [shift={(4,-4)},line width=2pt] (6,0)--(8,0);
\node [shift={(4,-4)},left] at (6,0) {$b_x$};
\draw [shift={(4,-4)},line width=2pt] (5,.8)--(7,.8);
\node [shift={(4,-4)},left] at (5,.8) {$b_{y'}$};
\draw [shift={(4,-4)},dotted, line width =1.5pt] (7,0)--(7,1.6);
\draw [shift={(4,-4)}, line width =2pt] (5.5,1.6)--(7.5,1.6);
\node [shift={(4,-4)},above] at (5.5,1.6) {$b_{y}$};
\draw [shift={(4,-4)}, line width =2pt] (5,1.3)--(7,1.3);
\node [shift={(4,-4)},left] at (5,1.3) {$b_{x'}$};
\node [shift={(4,-4)}] at (6,-.5) {Type 4: $xy$ is assigned to $b_{x'}$};
\end{tikzpicture}
\caption{The assignment of edges to bars from the proof of Theorem~\ref{thm:bipartitem/5}.
Each type can also occur with $b_x$ above $b_y$.}\label{fig:edgeassign}
\end{figure}

For edges of Types 3 and 4, we say that $b_{x'}$ {\it blocks} for $b_x$, since the presence of $b_{x'}$ allows $b_{x}$ to see two overlapping bars assigned to vertices in $Y$.
For Type 3, we say that $b_{x'}$ {\it blocks on the left} and for Type 4 we say that $b_{x'}$ {\it blocks on the right}.

It is clear that $b_x$ is assigned at most two edges of Type 1, since the left endpoint of $b_x$ can see at most two other bars.
It is also clear that $b_x$ can block on the left for at most one bar and can block on the right for at most one bar, so it is assigned at most one edge of Type 3 and at most one edge of Type 4.
If $b_x$ is able to see the left endpoint of two $Y$-bars above, then those bars must overlap.
Therefore, some bar blocks on the left for $b_x$, and that bar is assigned the edge joining $x$ to the $Y$-bar that is farther to the right (which is also the lower of the two).
Thus at most one of the upward visibilities of $b_x$ is assigned to it as an edge of Type 2.
Similarly, at most one of the downward visibilities of $b_x$ is assigned to it as an edge of Type 2.
Therefore $b_x$ is assigned at most two edges of Type 2.

Suppose that $b_x$ is assigned an edge of Type 2, and that this edge corresponds to a line of sight to a bar $b_y$ above $b_x$.
In this case, the left endpoint of $b_x$ cannot see the left endpoint of a $Y$-bar above $b_x$ as such a bar would overlap with $b_y$.
In this case, $xy$ would be an edge of Type $3$ that is assigned to some other bar.
Similarly, the right endpoint of $b_x$ cannot see the right endpoint of a $Y$-bar above $b_x$ since the right endpoint of $b_x$ sees an interior point of $b_y$ or a $Y$-bar that is to the right of and below $b_y$.
Therefore, if $b_x$ is assigned an edge of Type $2$ with a line of sight above $b_x$, then $b_x$ cannot be assigned an edge of Type 3 or 4 from an $X$-bar that is above $b_x$.
Similarly, if $b_x$ is assigned an edge of Type $2$ with a line of sight below $b_x$, then $b_x$ cannot be assigned an edge of Type 3 or 4 from an $X$-bar that is below $b_x$.
It follows that if $b_x$ is assigned two edges of Type 2, which requires one above and one below $b_x$, then $b_x$ cannot be assigned an edge of Type 3 or 4.
It follows that $b_x$ is assigned at most five edges.

Because $R$ is a $t$-unit-bar visibility representation of $K_{m,n}$ and each bar corresponding to a vertex in $X$ is assigned at most five edges, it follows that $mn\le 5nt$.
Therefore $t\ge \CL{\frac m5}$.
\end{proof}

When $n=2$, the proof of Theorem~\ref{thm:bipartitem/5} will actually yield an asymptotically best result.

\begin{theorem}\label{thm:Km2}
$ub(K_{2,m})=\frac m4+o(m)$.
\end{theorem}

\begin{proof}
The fact that $ub(K_{2,m})\le \frac m4+o(m)$ is established in Theorem~\ref{thm:KmnUB}.
To prove the lower bound we follow the proof of Theorem~\ref{thm:bipartitem/5}.
We claim that each $X$-bar is assigned at most four edges.
Assume to the contrary that $b$ is an $X$-bar that is assigned five edges.
It follows that $b$ is assigned two edges of Type 1, one edge of Type 2, one edge of Type 3, and one edge of Type 4.
Since a bar cannot be assigned an edge of Type 2 above (respectively below) and block for a bar that is above (respectively below), it follows that $b$ blocks for two $X$-bars that are both above or both below $b$.
Without loss of generality, assume that $b$ blocks for two $X$-bars $b'$ and $b''$ that are below $b$.
It follows that $b$, $b'$, and $b''$ all see the same $Y$-bar $\hat b$ and that $b\hat b$, $b'\hat b$, and $b''\hat b$ are all edges that are assigned to $b$.
However, since $n=2$, two of these three edges are the same, contradicting the assumption that each edge is assigned to a single pair of bars.
Therefore each $X$-bar in a representation of $K_{2,n}$ is assigned at most four bars and it follows that $ub(K_{2,m})\ge \CL{m/4}$.
\end{proof}

We now give another lower bound on the unit bar visibility number of complete bipartite graphs.
This bound is better than the bound from Theorem~\ref{thm:bipartitem/5} when $n\ge \frac 23 m$ and is asymptotically best possible when $n=m-o(m)$.
We note that this bound is an immediate consequence of Lemma~4 from~\cite{CHJLW}; we include the proof because it is conceptually different from those presented so far.

\begin{theorem}\label{thm:KmnLBBipartite}
For $m\ge n\ge 2$, $$ub(K_{m,n})\ge \frac{n}{2(m+n)}m+\frac{2}{m+n}.$$
Therefore, if $n=m-o(m)$, then $ub(K_{m,n})\ge \frac m4+o(m)$.
\end{theorem}

\begin{proof}
Let $R$ be a $t$-unit-bar visibility layout of $K_{m,n}$.
Represent each visibility in $R$ as a line segment joining two bars and then contract each bar to a point.
This yields a planar graph.
Because $R$ is a representation of a bipartite graph, the resulting planar graph is also bipartite.
Since there are at most $t(m+n)$ bars in the layout, there are at most $2t(m+n)-4$ edges in the representation.
Since $K_{m,n}$ has $mn$ edges, it follows that 
\begin{align*}
t&\ge\CL{\frac{mn+4}{2(m+n)}}\\
&\ge \frac{n}{2(m+n)}m+\frac{2}{m+n}.
\end{align*}
If $n=m-o(m)$, then it follows that 
\begin{align*}
t&\ge \frac{m^2}{4m-o(m)}-\frac{o(m^2)}{4m-o(m)}+\frac{2}{2m-o(m)}\\
&=\frac m4-o(m).\qedhere
\end{align*}
\end{proof}

Theorems~\ref{thm:KmnUB} and~\ref{thm:KmnLBBipartite} give an asymptotically sharp result for $K_{m,n}$ when $m$ and $n$ are asymptotic.

\begin{corollary}\label{cor:Kmn}
If $n=m-o(m)$, then $ub(K_{m,n})=\frac m4+o(m)$.
\end{corollary}

\section{Complete Graphs}\label{complete}

In~\cite{CHJLW}, Chang et al.\ proved that $b(K_n)=\CL{n/6}$, using constructions derived from the solution to Heawood's empire problem~\cite{Heawood,J-R,Wessel}.
They also provided a simpler construction to prove $b(K_n)\le\CL{n/6}+1$.
Based on this simpler construction we prove the following bound on the unit bar visibility number of $K_n$.

\begin{theorem}
$\CL{\frac n6}\le ub(K_n)\le \CL{\frac{n+4}{6}}$.
\end{theorem}

\begin{proof}
The lower bound follows from the facts that $b(G)\le ub(G)$ for all graphs and $b(K_n)=\CL{\frac n6}$.
For the upper bound, we give a brief description of the construction of an $(m+1)$-bar visibility representation of $K_{6m}$ from~\cite{CHJLW}.
Our only modification is the trivial observation that this construction can be obtained using $(m+1)$-unit-bars.

Partition $V(K_{6m})$ into three sets of size $2m$, which we call $V_1
$, $V_2$, and $V_3$.
The complete graph $K_{2m}$ has a decomposition into $m$ copies of $P_{2m}$, which can be obtained by rotating a zig-zag path when the vertices are placed on a circle.
The unit bar visibility representation of $K_{6m}$ is then obtained by decomposing $K_{6m}$ into $3m$ copies of $P_{2m}\vee 2K_1$ (the {\it join} of two graphs, $G\vee H$, is the graph obtained from the disjoint union of $G$ and $H$ by adding all edges joining $V(G)$ and $V(H)$).
Each copy of $P_{2m}\vee 2K_1$ will consist of a copy of $P_{2m}$ from the path-decomposition of $G[V_i]$ and two vertices from $V_{i-1\pmod 3}$ for some $i\in [3]$.
Provided that the sets of two vertices from $V_{i-1}$ are pairwise disjoint, each vertex lies in $m+1$ of the copies of $P_{2m}\vee 2K_1$.
The disjoint union of $3m$ copies of the unit bar visibility representation of $P_{2m}\vee 2K_1$ in Figure~\ref{fig:K_6m} with appropriate vertex-labels for the bars completes the construction.

\begin{figure}
\centering
\begin{tikzpicture}
\draw[line width=2pt] (.05,0)--(3.05,0);
\draw[line width=2pt] (-1.3,.2)--(1.7,.2);
\draw[line width=2pt] (-1,.4)--(2,.4);
\draw[line width=2pt] (-.7,.6)--(2.3,.6);
\draw[line width=2pt] (-.4,.8)--(2.7,.8);
\draw[line width=2pt] (-.1,1)--(2.9,1);
\draw[line width=2pt] (.2,1.2)--(3.2,1.2);
\draw[line width=2pt] (-1.15,1.4)--(1.85,1.4);

\end{tikzpicture}
\caption{A unit bar visibility representation of $P_{2m}\vee 2K_1$ with $m=3$.}\label{fig:K_6m}
\end{figure}
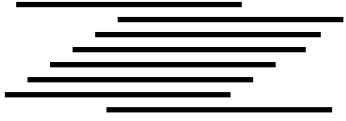

When $n$ is not divisible by $6$, a representation of $K_n$ can be obtained from the representation of $K_{6\CL{n/6}}$ by deleting the bars assigned to vertices not in $K_n$.
Therefore $ub(K_n)\le \CL{n/6}+1$ for all $n$.
It remains to show that we can do improve this bound by $1$ when $n\equiv 1\pmod 6$ or $n\equiv 2\pmod 6$.
We will assume that $n\equiv 2\pmod 6$ since a representation of $K_{6m+1}$ can be obtained from a representation of $K_{6m+2}$ by deleting the bars assigned to the extra vertex.
The reader is advised to consult Figure~\ref{fig:K14} to see an explicit example of the construction.

Let $n=6m+2$.
We begin with a decomposition of $K_{6m}$ into $3m$ copies of $P_{2m}\vee 2K_1$.
For $i\in [3]$, let $V_i=\{v_1^i,\ldots,v_m^i\}$.
Given $i\in [3]$ and $j\in [m]$, let $H_{i,j}$ be a copy of $P_{2m}\vee 2K_1$ where
\begin{enumerate}
\item the path $P_{2m}$ is the zigzag path on $V_i$ with endpoints $v_j^i$ and $v_{m+j}^i$, and
\item the copy of $2K_1$ has vertex set $\{v_j^{i-1\pmod 3},v_{m+j}^{i-1\pmod 3}\}$.
\end{enumerate}
For each $H_{i,j}$ we construct a copy of the representation in Figure~\ref{fig:K_6m}, where
\begin{enumerate}
\item $v_j^i$ is assigned to the second bar from the bottom (the bottom bar of the path);
\item $v_{m+j}^i$ is assigned to the second bar from the top (the top bar in the path);
\item $v_{j}^{i-1\pmod 3}$ is assigned to the bottom bar in the representation;
\item $v_{m+j}^{i-1\pmod 3}$ is assigned to the top bar in the representation.
\end{enumerate}

Note that in Figure~\ref{fig:K_6m} the second bar from the bottom is to the left of the top bar, and the second bar from the top is to the right of the bottom bar.
Furthermore, it is possible to compress the arrangement so that 1) the second bar from the bottom is to the left of the top bar by exactly $\frac{1}{12m}$, and 2) the second bar from the top is to the right of the bottom bar by exactly $\frac{1}{12m}$.

To complete the construction, we stack the representations of $P_{2m}\vee 2K_1$ so that every bar in the representation of $H_{i,j}$ is below every bar in the representation of $H_{i',j'}$ when $i<i'$, or $i=i'$ and $j<j'$.
Furthermore, for $i\in[3]$ and $j\in [m-1]$ arrange the representations of $H_{i,j}$ and $H_{i,j+1}$ so that the second bar from the bottom of $H_{i,j+1}$ is to the right of the topmost bar of $H_{i,j}$ by exactly $\frac{1}{12m}$.
Similarly, for $i\in[2]$ arrange the representations of $H_{i,m}$ and $H_{i+1,1}$ so that the second bar from the bottom of $H_{i+1,j}$ is to the right of the topmost bar of $H_{i,m}$ by exactly $\frac{1}{12m}$.
This yields a representation of $K_{6m}$ in which a single bar can be added above all other bars that will see the highest and second lowest bars of $H_{i,j}$ for all $i\in [3]$ and $j\in [m]$.
Similarly a single bar can be added below all other bars that will see the lowest and second highest bars of $H_{i,j}$ for all $i\in [3]$ and $j\in [m]$.
By construction, these bars will both see bars assigned to each vertex in $V_1\cup V_2\cup V_3$.
Letting $K_n$ have vertex set $V_1,\cup V_2\cup V_3\cup\{x,y\}$, we place a bar for $x$ and a bar for $y$ in those positions.
The addition of an extra bar for $x$ that sees the bar for $y$ finishes the construction.
\end{proof}

\begin{figure}
\centering
\begin{tikzpicture}
\draw[line width=2pt] (1.25,-.25)--(7.25,-.25);
\node [left] at (1.25,-.25) {$9$};
\draw[line width=2pt] (0,0)--(6,0);
\node [left] at (0,0) {$1$};
\draw[line width=2pt] (.5,.25)--(6.5,.25);
\node [left] at (.5,.25) {$2$};
\draw[line width=2pt] (1,.5)--(7,.5);
\node [left] at (1,.5) {$4$};
\draw[line width=2pt] (1.5,.75)--(7.5,.75);
\node [left] at (1.5,.75) {$3$};
\draw[line width=2pt] (.25,1)--(6.25,1);
\node [left] at (.25,1) {$11$};
\draw [decorate,decoration={brace,amplitude=4pt},xshift=0cm,yshift=0pt]
      (-1,-.25) -- (-1,1) node [midway,left] {$H_{1,1}$};
\draw [shift={(.5,1.75)},line width=2pt] (1.25,-.25)--(7.25,-.25);
\draw [shift={(.5,1.75)},line width=2pt] (0,0)--(6,0);
\draw [shift={(.5,1.75)},line width=2pt] (.5,.25)--(6.5,.25);
\draw [shift={(.5,1.75)},line width=2pt] (1,.5)--(7,.5);
\draw [shift={(.5,1.75)},line width=2pt] (1.5,.75)--(7.5,.75);
\draw [shift={(.5,1.75)},line width=2pt] (.25,1)--(6.25,1);
\node [shift={(.5,1.75)},left] at (.25,1)     {$12$};
\node [shift={(.5,1.75)},left] at (1.5,.75)   {$4$};
\node [shift={(.5,1.75)},left] at (1,.5)      {$1$};
\node [shift={(.5,1.75)},left] at (.5,.25)    {$3$};
\node [shift={(.5,1.75)},left] at (0,0)       {$2$};
\node [shift={(.5,1.75)},left] at (1.25,-.25) {$10$};
\draw [shift={(0, 1.75)},decorate,
		decoration={brace,amplitude=4pt},xshift=0cm,yshift=0pt]
      (-1,-.25) -- (-1,1) node [midway,left] {$H_{1,2}$};
\draw [shift={(1,3.5)},line width=2pt] (1.25,-.25)--(7.25,-.25);
\draw [shift={(1,3.5)},line width=2pt] (0,0)--(6,0);
\draw [shift={(1,3.5)},line width=2pt] (.5,.25)--(6.5,.25);
\draw [shift={(1,3.5)},line width=2pt] (1,.5)--(7,.5);
\draw [shift={(1,3.5)},line width=2pt] (1.5,.75)--(7.5,.75);
\draw [shift={(1,3.5)},line width=2pt] (.25,1)--(6.25,1);
\node [shift={(1,3.5)},left] at (.25,1)     {$3$};
\node [shift={(1,3.5)},left] at (1.5,.75)   {$7$};
\node [shift={(1,3.5)},left] at (1,.5)      {$8$};
\node [shift={(1,3.5)},left] at (.5,.25)    {$6$};
\node [shift={(1,3.5)},left] at (0,0)       {$5$};
\node [shift={(1,3.5)},left] at (1.25,-.25) {$1$};
\draw [shift={(0, 3.5)},decorate,
		decoration={brace,amplitude=4pt},xshift=0cm,yshift=0pt]
      (-1,-.25) -- (-1,1) node [midway,left] {$H_{2,1}$};
\draw [shift={(1.5,5.25)},line width=2pt] (1.25,-.25)--(7.25,-.25);
\draw [shift={(1.5,5.25)},line width=2pt] (0,0)--(6,0);
\draw [shift={(1.5,5.25)},line width=2pt] (.5,.25)--(6.5,.25);
\draw [shift={(1.5,5.25)},line width=2pt] (1,.5)--(7,.5);
\draw [shift={(1.5,5.25)},line width=2pt] (1.5,.75)--(7.5,.75);
\draw [shift={(1.5,5.25)},line width=2pt] (.25,1)--(6.25,1);
\node [shift={(1.5,5.25)},left] at (.25,1)     {$4$};
\node [shift={(1.5,5.25)},left] at (1.5,.75)   {$8$};
\node [shift={(1.5,5.25)},left] at (1,.5)      {$5$};
\node [shift={(1.5,5.25)},left] at (.5,.25)    {$7$};
\node [shift={(1.5,5.25)},left] at (0,0)       {$6$};
\node [shift={(1.5,5.25)},left] at (1.25,-.25) {$2$};
\draw [shift={(0,5.2)},decorate,
		decoration={brace,amplitude=4pt},xshift=0cm,yshift=0pt]
      (-1,-.25) -- (-1,1) node [midway,left] {$H_{2,2}$};
\draw [shift={(2,7)},line width=2pt] (1.25,-.25)--(7.25,-.25);
\draw [shift={(2,7)},line width=2pt] (0,0)--(6,0);
\draw [shift={(2,7)},line width=2pt] (.5,.25)--(6.5,.25);
\draw [shift={(2,7)},line width=2pt] (1,.5)--(7,.5);
\draw [shift={(2,7)},line width=2pt] (1.5,.75)--(7.5,.75);
\draw [shift={(2,7)},line width=2pt] (.25,1)--(6.25,1);
\node [shift={(2,7)},left] at (.25,1)     {$7$};
\node [shift={(2,7)},left] at (1.5,.75)   {$11$};
\node [shift={(2,7)},left] at (1,.5)      {$12$};
\node [shift={(2,7)},left] at (.5,.25)    {$10$};
\node [shift={(2,7)},left] at (0,0)       {$9$};
\node [shift={(2,7)},left] at (1.25,-.25) {$5$};
\draw [shift={(0, 7)},decorate,
		decoration={brace,amplitude=4pt},xshift=0cm,yshift=0pt]
      (-1,-.25) -- (-1,1) node [midway,left] {$H_{3,1}$};
\draw [shift={(2.5,8.75)},line width=2pt] (1.25,-.25)--(7.25,-.25);
\draw [shift={(2.5,8.75)},line width=2pt] (0,0)--(6,0);
\draw [shift={(2.5,8.75)},line width=2pt] (.5,.25)--(6.5,.25);
\draw [shift={(2.5,8.75)},line width=2pt] (1,.5)--(7,.5);
\draw [shift={(2.5,8.75)},line width=2pt] (1.5,.75)--(7.5,.75);
\draw [shift={(2.5,8.75)},line width=2pt] (.25,1)--(6.25,1);
\node [shift={(2.5,8.75)},left] at (.25,1)     {$8$};
\node [shift={(2.5,8.75)},left] at (1.5,.75)   {$12$};
\node [shift={(2.5,8.75)},left] at (1,.5)      {$9$};
\node [shift={(2.5,8.75)},left] at (.5,.25)    {$11$};
\node [shift={(2.5,8.75)},left] at (0,0)       {$10$};
\node [shift={(2.5,8.75)},left] at (1.25,-.25) {$6$};
\draw [shift={(0, 8.75)},decorate,
		decoration={brace,amplitude=4pt},xshift=0cm,yshift=0pt]
      (-1,-.25) -- (-1,1) node [midway,left] {$H_{3,2}$};

\draw [line width=2pt] (5,-.75)--(11,-.75);
\node [left] at (5,-.75)     {$14$};
\draw [line width=2pt] (5,-1.25)--(11,-1.25);
\node [left] at (5,-1.25)     {$13$};
\draw [line width=2pt] (-.75,10.25)--(5.25,10.25);
\node [left] at (-.75,10.25)     {$13$};
\end{tikzpicture}
\caption{A $3$-unit bar visibility representation of $K_{14}$.  Bars assigned to $v_j^i$ are labeled by the number $4(i-1)+j$.
Bars for $x$ are labelled 13, and the bar for $y$ is labelled 14.}\label{fig:K14}
\end{figure}
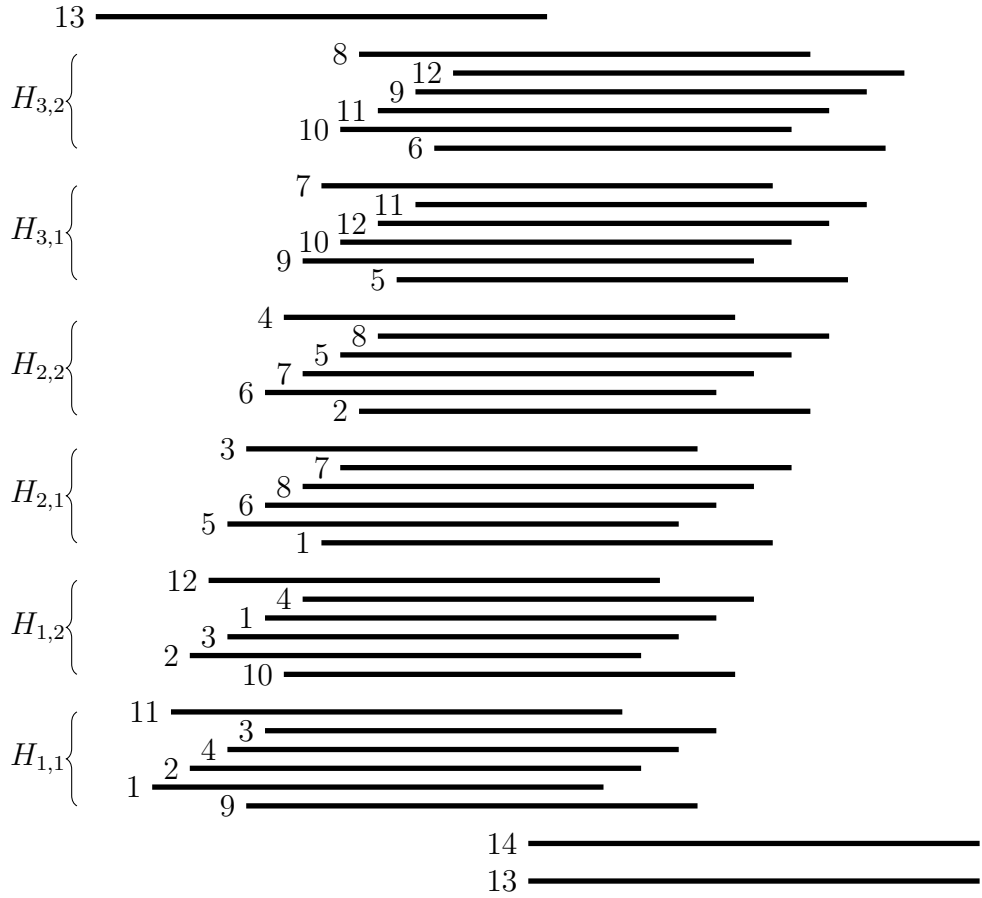

\section{Open Questions}\label{conclusion}

We conclude with some final remarks and open questions.
The first questions come from~\cite{DV} and~\cite{DE-MHP}.
\begin{question}
Is there a simple characterization of unit bar visibility graphs?
Is there a simple characterization of unit rectangle visibility graphs?
\end{question}

Regarding planar graphs, we ask if the bounds of Theorems~\ref{thm:planar} and~\ref{thm:planargirth} are sharp.
\begin{question}
Is there a planar graph $G$ such that $ub(G)=\CL{\frac{\Delta(G)+1}{3}}+2$?
Is there a planar graph $G$ with girth at least $7$ such that $ub(G)=\CL{\frac{\Delta(G)+1}{3}}+1$?
\end{question}

In light of Corollary~\ref{cor:Kmn} and Theorem~\ref{thm:Km2}, we conjecture that the upper bound on $ub(K_{m,n})$ is asymptotically sharp.
\begin{conjecture}
For $m\ge n\ge 2$, $ub(K_{m,n})=\frac m4+o(m)$.
\end{conjecture}

In~\cite{CHJLW}, Chang et al.\ proved that if $G$ is an $n$-vertex graph, then $b(G)\le\CL{n/6}+2$.
Lov\' asz~\cite{Lovasz} proved that every $n$-vertex graph can be decomposed into at most $\FL{n/2}$ paths and cycles, and since paths and cycles are both unit bar visibility graphs, we have the following general bound.

\begin{proposition}
If $G$ is an $n$-vertex graph, then $ub(G)\le \FL{n/2}$.
\end{proposition}

Among $n$-vertex graphs, the largest unit bar visibility number that we have found is for $K_{1,n-1}$, where $ub(K_{1,n-1})=\CL{(n-1)/3}$.
This leads to the following question.
\begin{question}
What is the largest value that $ub(G)$ can take when $G$ is an $n$-vertex graph?
In particular, is there an $n$-vertex graph $G$ for which $ub(G)>\CL{(n-1)/3}$?
\end{question}

We note that Algorithm~\ref{alg:tree} does not try to minimize the number of trees used in the decomposition.
It is natural to seek a decomposition into unit bar visibility forests that is in some sense more efficient.
\begin{question}
Is there an algorithm to efficiently determine the minimum number of elements in a decomposition of $T$ into UBVTs in which no vertex lies in more than $ub(T)$ of the elements?
\end{question}

Moving in a direction analogous to the idea of total interval numbers (see~\cite{AA}), we might seek unit bar visibility representations of graphs in which the total number of unit bars is minimized.

\begin{question}
Given a graph $G$, what is the minimum number bars in a unit bar visibility layout where the bars can be labeled by $V(G)$ giving a visibility representation of $G$?
\end{question}

\end{document}